\documentclass[12pt,leqno]{amsart}
\usepackage{latexsym,amssymb,amsfonts}
\usepackage{enumerate}
\usepackage{amsthm}
\usepackage{graphicx,color,wrapfig}
\usepackage{color}
\usepackage[font=scriptsize]{caption}
\definecolor{red}{rgb}{1.0,0.0,0.0}

\definecolor{blu}{rgb}{0.0,0.0,1.0}

\definecolor{violet}{rgb}{0.5,0.0,0.4}

\numberwithin{equation}{section}
\newtheorem{theo}{Theorem}[section]
\newtheorem{lemma}[theo]{Lemma}
\newtheorem{prop}[theo]{Proposition}
\newtheorem{cor}[theo]{Corollary}
\newtheorem{defi}[theo]{Definition}
\theoremstyle{definition}
\newtheorem{rem}[theo]{Remark}

\newtheorem{ass}[theo]{Assumptions}

\def\vertv{\vert_{_{\V}}}

\def\V{V^\prime}

\def\lipd{ ]\kern-1pt_{_{L}}}

\def\qn{ \hfill 
\hbox{\hskip 6pt\vrule width6pt height7pt
depth1pt  \hskip1pt}}

\begin{document}
\title[Equilibrium points for Optimal Investment]{Optimal Investment with
Vintage Capital:\\
Equilibrium Distributions}
\author[S. Faggian]{Silvia Faggian [1]}
\address{[1]S. Faggian, Department of Economics, Universit\'a ``\emph{Ca'
Foscari}'' Venezia, Italy. \texttt{faggian@unive.it}.}
\author[F. Gozzi]{Fausto Gozzi [2]}
\address{[2] Dipartimento di Economia e Finanza, Universit\'a LUISS - ``Guido
Carli, I-00162, Roma, Italy. \texttt{fgozzi@luiss.it}.}
\author[P. Kort]{Peter M. Kort [3]}
\address{[3] CentER, Department of Econometrics \& Operations Research,
Tilburg University, P.O. Box 90153, 5000 LE Tilburg, The Netherlands;
Department of Economics, University of Antwerp, Prinsstraat 13, 2000 Antwerp
1, Belgium, \texttt{kort@uvt.nl}.}
\maketitle

\begin{abstract}
The paper concerns the study of equilibrium points, or steady states, of
economic systems arising in modeling optimal investment with \textit{vintage
capital}, namely, systems where all key variables (capitals, investments,
prices) are indexed not only by time $\tau$ but also by age $s$. Capital
accumulation is hence described as a partial differential equation (briefly,
PDE), and equilibrium points are in fact equilibrium distributions in the
variable $s$ of ages. Investments in frontier as well as non-frontier
vintages are possible. Firstly a general method is developed to compute and
study equilibrium points of a wide range of infinite dimensional, infinite
horizon boundary control problems for linear PDEs with convex criterion,
possibly applying to a wide variety of economic problems. Sufficient and
necessary conditions for existence of equilibrium points are derived in this
general context. In particular, for optimal investment with vintage capital,
existence and uniqueness of a long run equilibrium distribution is proved
for general concave revenues and convex investment costs, and analytic
formulas are obtained for optimal controls and trajectories in the long run,
definitely showing how effective the theoretical machinery of optimal
control in infinite dimension is in computing explicitly equilibrium
distributions, and suggesting that the same method can be applied in
examples yielding the same abstract structure. To this extent, the results
of this work constitutes a first crucial step towards a thorough
understanding of the behavior of optimal controls and trajectories in the
long run.

\bigskip

\noindent \textit{Key words}: Equilibrium Points; Equilibrium Distributions;
Vintage Capital Stock; Age-structured systems; Maximum Principle in Hilbert
spaces; Boundary control; Optimal Investment.

\bigskip \noindent \textit{Journal of Economic Literature}: C61, C62, E22
\end{abstract}

\tableofcontents


\newpage

\section{Introduction}

Computing equilibrium points, or steady states, and describing their properties is one of the
main goals in the mathematics of economic models. This task, when presuming
an underlying optimal control problem with infinite horizon, is already
nontrivial with one state variable, but it becomes harsh when the dynamics
of the system are infinite dimensional, like in cases when
heterogeneity/path dependency is taken into account. This is the case, for
instance, of optimal investment with vintage capital (capital stock is
heterogeneous in age, see \textit{e.g.} {\cite{F3,FGJET}}), of spatial
growth models (capital stock is heterogeneous in space, see \textit{e.g.} {%
\cite{BoucekkineCamachoFabbri13, Fabbri16, BoucekkineFabbriFedericoGozzi18}}%
), of growth models with time-to-build (capital stock is path dependent, see
\textit{e.g.} {\cite{AseaZak99,Bambi08,BambiFabbriGozzi12}}), {or of models
with heterogeneous agents (see \textit{e.g.} \cite{MollNuno18}).} In all
these examples, equilibrium points are indeed functions (of vintage, or
space, or age) and may be more properly referred to as ``equilibrium
distributions''. Up to now such equilibrium distributions have been studied
\textit{only when the value function of the control problem is described by
an analytic formula} -- a requirement which is very seldom met -- so that
many interesting cases are left out of the picture.

\medskip

On the contrary, this work addresses the study of equilibrium {distributions}
in cases where no explicit formula for the value function is available,
moreover it does so under the general assumptions of an infinite-horizon
infinite-dimensional control problem with linear state equation and general
convex (concave, in the application) payoff, providing a theoretical tool
that can be used in {a variety} of applied examples. In fact, the theory is
put immediately into practise for the optimal investment model with vintage
capital, obtaining analytic formulas for the {equilibrium distributions},
and a complete sensitivity analysis for some instances of the problem. Hence
the paper contains a theoretical and an applied part, both of equal weight
and dignity, whose main achievements are listed below.

\medskip

For the general theory (Sections 2, 3 and 4), we reprise and complete the
study of the control problems analysed  in Faggian and Gozzi \cite{FaGo2}.
\footnote{%
Note that in this theoretical context, since equilibrium distributions can
be seen as points in a suitable infinite dimensional vector space, a space
of functions, they will be still named ``equilibrium points"} There Dynamic
Programming (DP) was employed to prove the existence and uniqueness of a
regular solution $v$ of the Hamilton-Jacobi-Bellman (HJB) equation, as well
as a verification theorem implying existence and uniqueness of optimal
feedback controls, and the fact that $v$ coincides with the value function.
Overall and differently from most contributions to the subject, this work
presents an integrated approach between the DP and MP methods of optimal
control theory. In particular:

\begin{itemize}
\item[(a)] a co-state is associated to the state variable, and necessary and
sufficient conditions for the optimal path are established in the form of a
Maximum Principle (MP) (Theorem \ref{Pmp});

\item[(b)] the co-state associated to an optimal state is shown to coincide
with the spatial gradient of the value function evaluated at that optimal
state (Theorem \ref{pi^*});

\item[(c)] the definition of two types of equilibrium points is introduced:
the stationary solutions of the state-costate system, called MP-equilibrium
points, and the stationary solutions of the closed loop equation (CLE)
arising in the DP approach (Definition \ref{def:eq}), called CLE-equilibrium
points;

\item[(d)]the relationship between the  two types of equilibrium points is explained, and
sufficient (and necessary) conditions for existence of such equilibria are
provided (Theorem \ref{th:equiv});

\item[(e)] two results on the stability of CLE-equilibrium points are given
by using or adapting the existing literature (Propositions \ref{pr:stability}
and \ref{pr:stabilitybis}).
\end{itemize}

It is important noting that the theory cannot be used straightforwardly to
treat applied problems in a satisfactory way. This happens on the one hand
because the results in infinite dimension need to be translated into terms
of the application under analysis, and on the other hand as it may be
necessary to exploit the particular structure of the applied problem to
specify formulas for practical use. One example is worked out through
Theorem \ref{equgen}, in the case of the model of optimal investment with
vintage capital.

\medskip

In the applied part of this work (Sections 5 and 6), the theoretical results
are used on the optimal investment model with vintage capital deriving:

\begin{itemize}
\item[(e)] the existence of {MP- or CLE-equilibrium points}, which is proven
equivalent to the existence of solutions of a numerical equation explicitly
derived from the data;

\item[(f)] analytic formulas for {MP- or CLE-equilibrium points} in some
relevant examples;

\item[(g)] a sensitivity analysis {for some particular sets of data}.
\end{itemize}

In particular, the sensitivity analysis enables the development of new
economic results while analyzing the vintage capital stock model in which
revenue is a strictly concave and linear quadratic function of output, where the strict
concavity is caused by market power on the output market. As is standard in
this literature (Feichtinger et al. \cite{F3}), output linearly depends on
the capital goods, whereas investment costs are convex and linear quadratic.
We show that the equilibrium distribution capital stock is first increasing and then
decreasing in the age of the capital good. The increasing part is the result
of investment costs being relatively large when capital goods are relatively
new. On the other hand, such investments are attractive due to the long
lifetime of new capital goods. Capital goods of older age have a shorter
lifetime. This gives an incentive to reduce investments in older capital
goods, resulting in the fact that the equilibrium distribution capital stock for old
machines decreases with respect to age. We further establish another
non-monotonicity dependence of the equilibrium distribution capital goods level, but now
with respect to the productivity of the capital goods. If productivity is
relatively low, the number of capital goods increases if productivity goes
up. This is because a given capital good produces more so that the firm is
more eager to invest in it. On the other hand, if productivity is relatively
large the firm decreases investments, because otherwise the firm
overproduces resulting in a too low marginal revenue. In other words, some
optimal output level exists and less capital goods are needed to produce
this level when productivity is high.

\medskip

In conclusion, this work shows how successfully and effectively the
theoretical machinery of optimal control in infinite dimension is in
computing \textit{explicit} formulas and studying properties for equilibrium
distributions, also \textit{in absence of an explicit formula for the value
function.} We believe that the theoretical tools developed in the first part
of this work can be {successfully} employed in examples yielding the same
abstract structure (like those mentioned at the beginning of this
introduction) and possibly extended to more complex cases with the use of
suitable numerical approximations: this will be the subject of future work.

\medskip

The paper is organized as follows.
 Section \ref{MODEL} {presents a family of
optimal investment models with vintage capital.} Sections \ref{sec:math}
presents the abstract optimal control problem and shows that the problem
contained in Section \ref{MODEL} falls into that wider class. {Section \ref%
{sec:mathbis} is the theoretical core of the paper, where we recall the
results obtained with the DP approach in \cite{FaGo2} (Section \ref{DPKNOWN}%
), we state and prove first order optimality conditions in terms of a
Maximum Principle (Section \ref{subsec:MP}), and we present and discuss the
general results on equilibrium points (Section \ref{SS:EQPT}).}  In Section %
\ref{sec:avc}, the general results of the previous sections are applied to
the model of optimal investment with vintage capital, providing a technique
to derive analytic formulas for the {equilibrium distributions}. Finally, in
Section \ref{sens}, a sensitivity analysis is conducted on some instances of
the problem of Section \ref{sec:avc}, i.e. where both revenues and costs are
chosen linear-quadratic. This section also contains numerical results as
illustration. An appendix with proofs of the theorems of Section \ref%
{sec:mathbis} and \ref{sec:avc}, as well as some additional results, completes the work.

 We remark that he paper is organized as to allow the reader less interested in mathematical details to approach Sections \ref{sec:avc} and \ref{sens}  without necessarily going through the theoretical Sections \ref{sec:math} and \ref{sec:mathbis}.

\bigskip

\subsection{Literature Review}

We complete this introductory section with an overview of literature on
vintage capital, and on optimal control of infinite dynamical systems,
thereby explaining what the present paper adds to each field.

From an economic point of view, the paper contributes to the literature of
vintage capital stock models. Such models extend standard capital
accumulation models, like, among many others, Eisner and Strotz \cite%
{Eisner63} and Davidson and Harris \cite{Davidson81} where capital goods are
a function of just time. The extension is that also the age of the capital
goods is taken into account. This enables to distinguish different vintages
of capital goods so that one could explicitly analyze issues like aging
(Barucci and Gozzi \cite{BG2}), learning (Greenwood and Jovanovic \cite%
{Greenwood01}), pollution (Xepapadeas and De Zeeuw \cite{Xepapadeas99}),
forest management (Fabbri, Faggian and Freni \cite{FFF}), and technological
progress (Feichtinger et al. \cite{F3}). We consider the kind of vintage
capital stock models where investments in older capital goods are possible.
This distinguishes the framework to be considered from works like Solow et
al.\cite{Solow}, Malcomson \cite{Malcomson}, Benhabib and Rustichini \cite%
{Benhabib}, and Boucekkine et al. \cite%
{Boucekkine1,Boucekkine2,Boucekkine3,Boucekkine4}.

The first contribution in vintage capital literature, which consider models where investments in older capital goods are also possible, is  Barucci and Gozzi
\cite{BG1}. They consider the vintage capital stock framework where, as in
Feichtinger et al. \cite{F2}, revenue is linearly increasing in output,
implying that the output price is constant, and linear-quadratic investment
costs.\ Like in Feichtinger et al. \cite{F2}, they do derive equilibrium distribution
expressions for capital goods of different ages and corresponding
investments. The present paper generalizes these contributions by obtaining
the equilibrium distribution expression of the capital goods for a model with general
concave function.

Barucci and Gozzi \cite{BG2} extends Barucci and Gozzi \cite{BG1} by
considering technological progress, while in Xepapadeas and De Zeeuw \cite%
{Xepapadeas99} the production process produces emissions next to products.
Both papers keep the revenue linearly dependent on output. Provided an
 equilibrium distribution exists, which is not the case when we have ongoing technological
progress as in Barucci and Gozzi \cite{BG2}, due to this linearity equilibrium distribution expressions are much easier to obtain compared to a revenue function
being concave as in the present paper.

Closer to our present paper than the works cited above is Feichtinger et al.
\cite{F3}, in which also a firm with market power is considered. The
difference with our work is that Feichtinger et al. considers technological
progress. In particular, the main part of their work analyzes how the firm
reacts with its investment policy to a technological breakthrough, which is
a point in time at which a new technology is invented. The implication is
that productivity of the capital goods of vintages borne after the
breakthrough time jumps upwards. Our model is simpler in the sense that we
do not consider technological progress. However, our analysis goes further
than in Feichtinger et al. \cite{F3} in that we were able to derive an
analytical expression for the equilibrium distribution. This we could do for a general
concave revenue function, where Feichtinger et al. \cite{F3} just considers
linear-quadratic revenue. Note that after the technological breakthrough
Feichtinger's model turns into our model with prespecified revenue function.
This implies that also in their framework a unique equilibrium distribution exists,
which can be calculated using the results of the present paper.

\bigskip

From the point of view of mathematics, the main features of the optimal
control problem here considered are: $(i)$ the linear state equation and the
convex cost criterion; $(ii)$ the presence of a boundary control; $(iii)$
the age structure of the driving operator $A$ in the state equation.

Optimal control of infinite dimensional systems is the subject of many books
and papers in the recent literature. Among the books in the deterministic
case we mention Lions \cite{JLL} and Barbu and Da Prato \cite{BD1}, and the
more recent ones Li and Yong \cite{LiYong95}, and Troltzsch \cite%
{Troltzsch10book}. For the stochastic case (concerning the dynamic
programming approach) one can see the recent book \cite{FabbriGozziSwiech17}.

Concerning the dynamic programming approach to problems with linear state
equation and convex cost but with distributed control, we refer the reader
to Barbu and Da Prato \cite{BD1, BD2,BD3}, for some linear convex problems
to Di Blasio \cite{D1,D2}, for the case of constrained control to Cannarsa
and Di Blasio \cite{CD}, and for the case of state constraints to Barbu, Da
Prato and Popa \cite{BDP} (see also Gozzi \cite{G1,G2,G3} for a
generalization of this approach to the case of semilinear state equations).
For boundary control problems we recall, in the case of linear systems and
quadratic costs (where the HJB equation reduces to the operator Riccati
equation) \textit{e.g.} the books by Lasiecka and Triggiani \cite{LT,LTold},
the book by Bensoussan, Da Prato, Delfour and Mitter \cite{BDDM}, and, for
nonautonomous systems, the papers by Acquistapace, Flandoli and Terreni \cite%
{AFT, AT1, AT2, AT3}. For the case of a linear system and a general convex
cost function, we mention the papers by Faggian \cite{Fa1,Fa0,Fa2,Fa3,FaSC},
and by Faggian and Gozzi \cite{FaGo,FaGo2} (in particular, the theory
developed in the last two works is the starting point for theory in the
present paper, and is  recalled in Section \ref{DPKNOWN}). On the Pontryagin
maximum principle for boundary control problems we mention again, in the
linear quadratic case, the books \cite{LT,LTold,JLL} and \cite{BDDM}; in the
case of linear systems with convex cost, e.g., the book by Barbu and
Precupanu (Chapter 4 in \cite{BP}), and the papers \cite{Barbu80}, \cite%
{Iftode89}; for general nonlinear boundary control problems, e.g., \cite%
{CannarsaTessitore94}, \cite{Fattorini68}, \cite{Fattorini87}, \cite%
{Troltzsch89} \cite{GozziTessitore97} \cite{GozziTessitore98}.  None of them
covers the class of problems treated here.

The main contributions of the present paper with respect to the mathematical
literature quoted above are: (1) the proof of the Maximum Principle for
infinite dimensional, infinite horizon optimal control problems with
features $(i)-(iii)$; (2) the co-state inclusion which reconnects the value
function with the co-state; (3) the analysis of equilibrium points of the
control problem.

\section{The optimal investment model with vintage capital}

\label{MODEL}

We now describe the model of optimal investment with vintage capital, in the
setting introduced by Barucci and Gozzi \cite{BG1}\cite{BG2}, and later
reprised and generalized by Feichtinger et al. \cite{F1,F2,F3}, and by
Faggian \cite{Fa2,Fa3} and Faggian and Gozzi \cite{FaGo}.

The capital accumulation process is given by the following system
\begin{equation}
\begin{cases}
\frac{\partial K(\tau ,s)}{\partial \tau }+\frac{\partial K(\tau ,s)}{%
\partial s}+\mu K(\tau ,s)=u_{1}(\tau ,s),\quad (\tau ,s)\in ]t,+\infty
\lbrack \times ]0,\bar{s}] \\
K(\tau ,0)=u_{0}(\tau ),\quad \tau \in ]t,+\infty \lbrack \\
K(t,s)=x(s),\quad s\in \lbrack 0,\bar{s}]%
\end{cases}
\label{ipde}
\end{equation}%
with $t>0$ the initial time, $\bar{s}\in \lbrack 0,+\infty ]$ the maximal
allowed age, and $\tau \in \lbrack 0,T[$ with horizon $T=+\infty $. The
unknown $K(\tau ,s)$ represents the amount of capital goods of age $s$
accumulated at time $\tau $, the initial datum is a function $x\in L^{2}(0,%
\bar{s})$ (the space of square integrable functions on $(0,\bar{s})$), $\mu
>0$ is a depreciation factor. Moreover, $u_{0}:[t,+\infty \lbrack
\rightarrow {{\mathbb{R}}}$ is the investment in new capital goods ($u_{0}$
is the boundary control) while $u_{1}:[t,+\infty \lbrack \times \lbrack 0,%
\bar{s}]\rightarrow {{\mathbb{R}}}$ is the investment at time $\tau $ in
capital goods of age $s$ (hence, the distributed control). Investments are
jointly referred to as the control $u=(u_{0},u_{1})$. The output rate is
\begin{equation}
{Q(K(\tau )):}=\int_{0}^{\bar{s}}\alpha (s){K(\tau ,s)}ds,  \label{Q}
\end{equation}%
where $\alpha (s)$ is a productivity parameter. Selling the output to
consumers results in an instantaneous revenue, $R\left( Q\right) ,$ where $R$
is a concave function. Capital stock can be increased by investing, and
investment costs are given by
\begin{equation}
C({u(\tau )})\equiv C_{0}({u}_{0}{(\tau )})+C_{1}({u}_{1}{(\tau )})\equiv
C_{0}(u_{0}({\tau )})+\int_{0}^{\bar{s}}{c_{1}}(s,u_{1}({\tau ,}s))ds,
\label{ccc}
\end{equation}%
with $C_{1}$ indicating the investment cost rate for technologies of age $s$%
, $C_{0}$ the investment cost in new technologies, including adjustment
costs, $C_{0}$, $C_{1}$ convex in the control variables. The firm's payoff
is then represented by the functional
\begin{equation}
I(t,x;u_{0},u_{1})=\int_{t}^{+\infty }e^{-\lambda \tau }[R({Q(K(\tau )})-C({%
u(\tau )})]d\tau ,  \label{profits}
\end{equation}%
where $\lambda \in \mathbb{R}$ is the discount rate. Note that $\lambda$
is usually assumed positive, but here we leave the possibility of choosing a
negative $\lambda$  (corresponding, for example, to a negative
interest rate). The entrepreneur's problem is that of maximizing $%
I(t,x;u_{0},u_{1}\kern-1pt)$ over all state--control pairs $\{K,(\kern%
-1ptu_{0},u_{1}\kern-1pt)\kern-1pt\}$ which are solutions (in a suitable
sense) of equation (\ref{ipde})   
and keep the capital stock $K(\tau,s)$  nonnegative at all times. Such a problem is known as
 \emph{vintage capital} problem, for
the capital goods depend jointly on time $\tau $ and on age $s$, which is
equivalent to their dependence on time and vintage $\tau -s$.

We finally recall the definition of the value function of the problem
\begin{equation}  \label{eq:VFsec2}
V(t,x):=\inf_{u\in L_\lambda^p(t,+\infty, U)}I(t,x;u_0,u_1).
\end{equation}
Since $R$ and $C$ are not time dependent it is immediate to see that
\begin{equation}  \label{eq:V0Fsec2}
V(t,x)=e^{-\lambda t}V(0,x)=:V_0(x).
\end{equation}

\begin{rem}
\label{sc1} As a matter of fact, we treat the above problem without the
state constraints $K(\tau,s)\ge0$ for all $s$ and $\tau$, and check that
constraints are satisfied \emph{a posteriori} by the optimal trajectories of
the unconstrained problem. In such a case, those trajectories are also optimal
for the problem with state constraints.\qn
\end{rem}

\subsection{Revenues and costs}

In order to be able to treat optimal investment with vintage capital into
the wider class of abstract problems described in Sections \ref{sec:math}
and \ref{sec:mathbis}, we specify the assumptions on revenues $R$ and costs $%
C$  which ensure that the basic assumptions of the abstract problem
(Assumptions \ref{asst2}, (3)-(6)) are fulfilled.

\begin{ass}
\label{asst1}

\begin{enumerate}
\item[($i$)] $R\in C^1(\mathbb{R})$, $R$ concave, $R^\prime$ Lipschitz
continuous. Moreover $\alpha \in H^1(0,\bar s)$\footnote{$H^1(0,\bar s)$ is
the space of square integrable functions which admit a square integrable
derivative in weak sense. Continuous functions with piecewise continuous
derivatives are included in this space.} and $\alpha(\bar s)=0$.

\item[($ii$)] $C_0(r)$ and $r\mapsto c_1(s,r)$ are convex, lower
semi--continuous functions, with injective{\footnote{%
A multivalued function $\rho:U\to\mathbb{R}$ is injective when $%
\rho(u_1)\cap\rho (u_2)=\emptyset$ for every $u_1,u_2\in U$, $u_1\ne u_2$.}}
subdifferential at all $r\in\mathbb{R}$.

\item[($iii$)] $C_0^*(r) ,r\mapsto c_1(s,\cdot)^*(r)$ (are Fr\'echet
differentiable and) have Lipschitz continuous derivatives, for all $s\in[%
0,\bar s]$.

\item[($iv$)] $C_0(r)$ and $r\mapsto c_1(s,r)$ are bounded below by a
function of type $a\vert r\vert^p+b$, for some $a>0$, $b\in \mathbb{R}$, $%
p>1 $.

\end{enumerate}
\end{ass}

In the above statement, we denoted by $f^{\ast }$ the convex conjugate of a
convex function $f$, in particular $C_{0}^{\ast }(r)=\sup_{w\in \mathbb{R}%
}\{wr-C_{0}(w)\},$ ${c_{1}}(s,\cdot )^{\ast }(r)=\sup_{w\in \mathbb{R}}\{wr-{%
c_{1}}(s,w)\}.$ Note that no strong regularity of $C$ is required.
\bigskip

For example, suitable choices for the \textit{revenues} are the following:

(a) \textit{Linear-quadratic: }$R(Q)=-aQ^{2}+bQ$;

(b) \textit{Logarithmic}: $R(Q)=\ln (1+Q),$ for $Q\geq 0$ and $R(Q)=Q$ for $%
Q<0$;

(c) \textit{Power} $\gamma \in (0,1)$: $R(Q)=b[(\nu +Q)^{\gamma }-\nu^\gamma
],$with $b,\nu >0$ ($\nu $ arbitrary small), for $Q\geq 0$ and $R(Q)=b\gamma
\nu^{\gamma -1}Q$ for $Q<0.$ Note in particular that this $R$ converges as $%
\nu $ tends to 0 to $R(Q)=bQ^{\gamma },$ for $Q\geq 0$ and $R(Q)=-\infty $
for $Q<0\footnote{%
The definition of $R$ for negative values of $Q$ is needed in order to apply
the general theory, although negative values of $Q$ will never emerge in our
calculations. Note also that setting $R\left( Q\right)=-\infty$ for $Q<0$ is
equivalent to require $Q\geq 0$ in optimal solutions.}.$

\bigskip

\bigskip Suitable choices for the \textit{costs} are, once set $\beta
=(\beta _{0},\beta _{1})\in \mathbb{R} \times L^{\infty }(0,\bar{s})$, with $%
\beta _{1}(s),\beta _{0}\geq \epsilon \geq 0$, $q=(q_{0},q_{1})\in
\mathbb{R}_+\times L^{2}(0,\bar{s})$, the following:

(A) \textit{Linear-quadratic}:
\begin{equation}  \label{lqcost}
C(u)=\int_{0}^{\bar{s}}[\beta _{1}(s)u_{1}^{2}(s)+q_{1}(s)u_{1}(s)]ds+\beta
_{0}u_{0}^{2}+q_{0}u_{0}
\end{equation}

(B) \textit{Linear+quadratic with constrained control:}
\begin{eqnarray}
C(u_{0},u_{1}) &=&C_{0}(u_{0})+C_{1}(u_{1})  \label{lqc} \\
&=&q_{0}u_{0}+g_{\beta _{0},M_{0}}(u_{0})+\int_{0}^{\bar{s}}\left[ \alpha
_{1}(s)u_{1}(s)+g_{\beta _{1}(s),M_{1}}(u_{1}(s))\right] ds
\end{eqnarray}%
where
\begin{equation*}
g_{\beta ,M}(u)=\left\{
\begin{array}{c}
\beta u^{2} \\
+\infty%
\end{array}%
\begin{array}{c}
|u|\leq M \\
|u|>M%
\end{array}%
\right.
\end{equation*}%
Such a cost can be easily generalized to a case where $u$ belongs to any
compact interval and not necessarily $u\in \lbrack -M,M]$.

(C) \textit{Linear+Power} \textit{costs}:
\begin{eqnarray}  \label{pc}
C(u_{0},u_{1}) &=&C_{0}(u_{0})+C_{1}(u_{1}) \\
&=&q_{0}u_{0}+f_{\beta _{0}}(u_{0})+\int_{0}^{\bar{s}}\left[
q_{1}(s)u_{1}(s)+f_{\beta _{1}(s)}(u_{1}(s))\right] ds
\end{eqnarray}%
where,  for $p>2$,
\begin{equation*}
f_{\beta }(u)=\left\{
\begin{array}{c}
\beta \left[ (u+\theta )^{p}-\theta ^{p}\right] \\
+\infty%
\end{array}%
\begin{array}{c}
u\geq 0 \\
u<0%
\end{array}%
\right.
\end{equation*}%
which implies also positivity constraints of the controls.

We treat all of these cases in Section \ref{sec:avc}. Moreover, in Section %
\ref{sens} we treat the case of linear--quadratic revenues and costs for
which we derive analytic formulas for the long run optimal couples, and
perform a complete sensitivity analysis.

\bigskip

The reader is advised that Sections \ref{sec:math}, \ref{sec:mathbis}, and
the Appendix are devoted to the mathematics of the general problem and
require a good knowledge of functional analysis to be fully understood.
Nonetheless, they may be skipped at a first reading, as the reader will find
in Section \ref{sec:avc} the theoretical results translated in terms of the
problem of optimal investment with vintage capital.

\section{The theoretical framework}

\label{sec:math}

Here we introduce an abstract class of infinite dimensional optimal control
problems with linear evolution equation and convex payoff, in which the
control may also act on the boundary, and address it as {(P)}. Then, in
Section \ref{INTROSUBMAT} we show that the optimal investment model with
vintage capital described in the previous section is of type (P).

\subsection{Notation}

\label{SS:NOT} The expression $a\vee b$ means the maximum of the real
numbers $a$ and $b$. If $X$ is a Banach space, we indicate its norm with $%
|\cdot |_{X}$, its dual with $X^{\prime }$, with $\langle \cdot ,\cdot
\rangle_{X^{\prime },X} $ the duality pairing. When $X=V^{\prime }$ we use
for simplicity $\langle \cdot ,\cdot \rangle$ in place of $\langle \cdot
,\cdot \rangle_{V^{\prime },V}$. If $X$ is also a Hilbert space, we indicate
with $(\cdot \vert\cdot )_{X} $ the inner product in $X$.

If $X$ and $Y$ are Banach spaces, then $C^{1}(X)$ denotes all Fr\'{e}chet
differentiable functions from $X$ to $\mathbb{R}$, and $\mathcal L(X,Y)$ the set of
all linear and continuous operators from $X$ to $Y$, with associated norm $\Vert \cdot\Vert_{\mathcal L(X,Y)}$. Moreover we set
\begin{equation*}
\begin{split}
& Lip(X;Y)=\{f:X\rightarrow Y~:~[f]\kern-1pt_{_{L}}:=\sup_{x,y\in X,~x\neq y}%
\frac{|f(x)-f(y)|_{Y}}{|x-y|_{X}}<+\infty \} \\
& C_{Lip}^{1}(X):=\{f\in C^{1}(X)~:~[f^{\prime }]\kern-1pt_{_{L}}<+\infty \}
\end{split}%
\end{equation*}%
{and, for $p\ge 1$,}
\begin{equation*}
\begin{split}
& \mathcal{B}_{p}(X,Y):=\{f:X\rightarrow \mathbb{R}~:~|f|_{\mathcal{B}%
_{p}}:=\sup_{x\in X}{\frac{|f(x)|_{Y}}{1+|x|_{X}^{p}}}<+\infty \},\ \ \
\mathcal{B}_{p}(X):=\mathcal{B}_{p}(X,\mathbb{R}), \\
& C([0,T],\mathcal{B}_{p}(X,Y)):=\{f:[0,T]\rightarrow \mathcal{B}_{p}(X,Y)\
:\ f\text{ continuous}\}
\end{split}%
\end{equation*}%
Note that $\mathcal{B}_{p}$ are Banach spaces if endowed with the norm $%
|\cdot |_{\mathcal{B}_{p}}$, so that continuity is intended with respect to
such norms. Furthermore, we set
\begin{equation*}
\Sigma _{0}(X):=\{w\in C_{Lip}^{1}(X)\ :\ w\ \mathrm{is\ convex}\}.
\end{equation*}%
Finally, if $X$ is a Hilbert space and $h:X\rightarrow \mathbb{R}$ is a
convex function, then $h^{\ast }$ will denote its convex conjugate, namely $%
h^{\ast }:X\rightarrow \mathbb{R}$, $h^{\ast }(x)=\displaystyle\sup_{y\in
X}\{\langle x,y\rangle -h(y)\}$. \bigskip


\subsection{The abstract optimal control problem (P)}

\label{SS:ABS}

We consider two real separable Hilbert spaces $V$ and $H$ with $V$
continuously embedded in $H$. We identify $H$ with its dual and we call $%
V^{\prime }$ the topological dual of $V$, which we do not identify with $V$
for the reasons explained in Section \ref{INTROSUBMAT}. We then get a
so-called Gelfand triple
\begin{equation*}
V \hookrightarrow H \hookrightarrow V^{\prime }
\end{equation*}
We choose $V^{\prime }$ as state space. The control space is the real
separable Hilbert space $U$ (which we identify with its dual $U^{\prime }$).
We consider the control system with state space $V^{\prime }$, control space
$U$, and varying initial time $t\ge 0$, described by
\begin{equation}
\begin{cases}
\label{eq:statoV'}y^{\prime }(\tau )=Ay(\tau )+Bu(\tau ), & \tau >t \\
y(t)=x\in V^{\prime }, &
\end{cases}%
\end{equation}
 where $A$ and $B$ are linear operators, 
possibly
unbounded. Moreover, we take a convex functional of the following type
\begin{equation}
J (t,x,u)=\int_{t}^{+\infty }e^{-\lambda \tau }\left[ g_{0}\left( y(\tau
)\right) +h_{0}\left( u(\tau )\right) \right] d\tau  \label{J in H}
\end{equation}%
where the function $g_{0}$ and $h_{0}$ are convex functions. The problem {(P)%
} is that of minimizing $J_{\infty }(t,x,u)$ with respect to $u$, over the
set of admissible controls
\begin{equation}
L_{\lambda }^{p}(t,+\infty ;U)=\{u:[t,+\infty )\rightarrow U\ ;\ \tau
\mapsto u(\tau )e^{-\frac{\lambda \tau }{p}}\in L^{p}(t,+\infty ;U)\},
\end{equation}
which is a Banach space with the norm
\begin{equation*}
\vert u\vert _{L_{\lambda }^{p}(t,+\infty ;U)}=\int_{t}^{+\infty }|u(\tau
)|_{U}^{p}e^{-\lambda \tau }d\tau = \vert e^{-\frac{{\lambda }}{p}(\cdot)}u {%
(\cdot)}\vert _{L^{p}(t,+\infty ;U)}.
\end{equation*}

\begin{rem}
\label{sc2} In the above problem no constraints on controls or on states are
assumed although, in economic applications, the state represents capital
stock, usually assumed nonnegative. Here we proceed along with the frequently  
  used idea (see e.g.\cite{FGJET}) to check {\emph{ex post}} that the
constraints are satisfied by the optimal trajectories of the unconstrained
problem, so those trajectories are optimal also for the constrained problem.\qn
\end{rem}

{The basic assumptions on the data are stated below and will hold throughout
the paper.}

\begin{ass}
\label{asst2}

\begin{enumerate}
\item[(1)] $A:D(A)\subset V^{\prime }\rightarrow V^{\prime }$ is the
infinitesimal generator of a strongly continuous semigroup $\{e^{\tau
A}\}_{\tau \geq 0}$ on $V^{\prime }$.  {Moreover there exists $\omega\in
\mathbb{R}$ such that\footnote{{When $\omega >0$, a semigroup $S(t)$ with
this property is usually called a \emph{pseudo-contraction semigroup}, as $%
e^{-\omega t}S(t)$ is a contraction semigroup with generator $A-\omega I$.}}
\begin{equation*}
\vert e^{\tau A}x\vert_{_{V^\prime}}\le e^{\omega \tau}\vert
x\vert_{_{V^\prime}},~\forall \tau\ge0;
\end{equation*}%
}

\item[(2)] $B\in L(U,V^\prime)$;

\item[(3)] $g_0\in\Sigma_0(V^\prime)$

\item[(4)] $h_0:U\to \mathbb{R}$ is convex and lower semi--continuous, $%
\partial h_0$ is injective.

\item[(5)] $h_0^*(0)=0$, $h_0^*\in \Sigma_0({U})$;

\item[(6)] $\exists a>0$, $\exists b\in \mathbb{R}$, $\exists p>1$ : $%
h_0(u)\ge a\vert u\vert_U^p+b$, $\forall u\in U$;  {}

\item[(7)] $\lambda>\omega$.
\end{enumerate}
\end{ass}

In proving some results we will need to specify the assumption (7) above as
follows.

\begin{ass}
\label{asst2bis} In addition to Assumption \ref{asst2}, we require that
either

\begin{itemize}
\item[(1)] $p> 2,$ \ \ $\lambda>(2\omega)\vee\omega$
\end{itemize}

or

\begin{itemize}
\item[(2)] $g_0 \in \mathcal{B}_1(V^\prime).$
\end{itemize}
\end{ass}

The adjoint of $A$ in the inner product of $V^{\prime }$ is denoted by $A^*$%
, while the adjoint of $A$ with respect to the duality $\langle\cdot,\cdot%
\rangle$ in $V,V^\prime$ is the unbounded operator $A_1^*$ on $V$, $A_1^*:
D(A_1^*)\subset V \to V.$

\begin{rem}
\label{coniugate} We recall that, if a function {$h:U\to {\mathbb{R}}\cup
\{+\infty\}$ is lower semicontinuous and convex (and not identically $%
+\infty $),} then the subgradient $\partial h$ is defined as $\partial
h(u)=\{u^*\in U~:~h(v)-h(u)\ge\langle u^*,v-u\rangle_U,~\forall v\in U\}.$
Moreover if $\partial h$ is injective
then $h^*$ is Fr\'echet differentiable with $(h^*)^\prime(u)=(\partial
h)^{-1}(u)$ for all $u\in U$.\qn
\end{rem}

\subsection{{Optimal investment with vintage capital is of type {(P)}} \label%
{INTROSUBMAT}}

We end the section by showing that the problem of optimal investment with
vintage capital described in Section \ref{MODEL} falls in the general class {%
(P)} described above and refer interested readers to \cite{Fa3} for full
detail.\footnote{%
See also \cite{EN}, \cite{BDDM} or \cite{Pazy} for the general theory of
strongly continuous semigroups and evolution equations.}

We at first formulate an intermediate abstract problem in $H=L^{2}(0,%
\overline{s})$, the space of square integrable functions of variable $s$,
using the modified translation semigroup $\{e^{A_{0}t}\}_{t\geq 0}$ on $H$,
namely the linear operators $e^{A_{0}t}:H\rightarrow H$ such that
\begin{equation*}
[e^{A_{0}t}f](s)=f\left( s-t\right) e^{-\mu t}, \quad \hbox{if $s\in \lbrack
t,\overline{s}]$,} \quad and \quad \hbox{$[e^{A_{0}t}f](s)=0$ otherwise.}
\end{equation*}
If $H^{1}(0,\bar{s})=\{f\in L^{2}(0,\overline{s}):f^{\prime }\in L^{2}(0,%
\overline{s})\},$ then the generator of $\{e^{A_{0}t}\}_{t\geq 0}$ is the
operator $A_{0}:D(A_{0})\subset H\rightarrow H,$ with
\begin{equation*}
D(A_{0})=\{f\in H^{1}(0,\bar{s}):f(0)=0\},\qquad A_{0}f(s)=-f^{\prime
}(s)-\mu f(s).
\end{equation*}
The adjoint of $A_{0}$ is then $A_{0}^{\ast }:D(A_{0}^{\ast })\rightarrow H$
with
\begin{equation*}
D(A_{0}^{\ast})=\{f\in H^{1}(0,\bar{s}):f(\bar{s})=0\}, \qquad
[A_{0}^{\ast}f](s)=f^{\prime }(s)-\mu f(s),
\end{equation*}
generating itself a modified translation semigroup $\{e^{A_{0}^{\ast
}t}\}_{t\geq 0}$ on $H,$ given by
\begin{equation*}
\left[ e^{A_{0}^{\ast }t}f\right] (s)=f\left( s+t\right) e^{-\mu t}, \quad %
\hbox{if $s\in\lbrack 0,\overline{s}-t]$,} \quad and \quad \hbox{$\left[
e^{A_{0}^{\ast }t}f\right] (s)=0$ otherwise}.
\end{equation*}
The control space is $U={{\mathbb{R}}}\times H$, the control function is a
couple
\begin{equation*}
u\equiv (u_{0},u_{1}):[t,+\infty )\rightarrow {{\mathbb{R}}}\times H,
\end{equation*}
and the control operator is given by
\begin{equation*}
Bu\equiv B(u_{0},u_{1})=u_{1}+u_{0}\delta _{0}, \quad \hbox{for all
$(u_{0},u_{1})\in {{\mathbb{R}}}\times H$,}
\end{equation*}
$\delta _{0}$ being the Dirac delta at the point $0$. \color{black} With
this notation, the original state equation (\ref{ipde}) can be written as%
\begin{equation}
\begin{cases}
K^{\prime }(\tau )=A_{0}K(\tau )+Bu(\tau ), & \tau >t \\
K(t)=x, &
\end{cases}
\label{statoH}
\end{equation}
Note that $H$ and $U$ are Hilbert spaces, and that $B$ is unbounded, meaning
that it is \emph{not} a continuous operator from $U$ to $H$ (unless $u_{0}=0$%
, corresponding to identically null boundary control $u_{0}(\tau )$), for
the Dirac delta does not lie in $H$. Then (\ref{statoH}) needs to be
interpreted in a suitable way, for instance in an extended state space.
\bigskip

Then we generalize all previous notions to a wider space. We set $V\equiv
D(A_{0}^{\ast }),$ and assume $V^{\prime }$ as state space of the abstract
problem. Indeed by standard arguments (see e.g. \cite[Section II.5]{EN}) --
and in particular by replacing the scalar product in $L^{2}$ with the
duality pairing $\left\langle \phi ,\psi \right\rangle $ with $\phi \in
V^{\prime }$, $\psi \in V$ (coinciding with the {inner} product in $L^{2}$
when $\phi \in L^{2}$) -- the semigroup $\{e^{A_{0}t}\}_{t\geq 0}$ can be
extended to a strongly continuous semigroup $\{e^{At}\}_{t\geq 0}$ on $%
V^{\prime }$, by setting
\begin{equation}  \label{eq:dualitynew}
\langle e^{At}\phi, f \rangle = \langle \phi,e^{A_0^*t}f \rangle\ \ \text{%
for\ every\ } f \in V, \phi \in V^{\prime },
\end{equation}
The generator of $\{e^{At}\}_{t\geq 0}$ is the operator $A:D(A)\subset
V^{\prime }\rightarrow V^{\prime },$ with $D(A)=H$. Moreover the semigroup $%
\{e^{A_{0}^{\ast }t}\}_{t\geq 0}$ can be restricted to a strongly continuous
semigroup on $V$, with generator the restriction of $A_{0}^{\ast }$ to $%
D\left((A_{0}^{\ast})^{2}\right)$.  Such restriction is exactly the adjoint
of $A$ in the duality $\langle\cdot,\cdot\rangle$ and is then denoted, as in
the previous subsection, by $A_{1}^{\ast }$.

The role of $H$ is that of pivot space between $V$ and $V^{\prime }$, namely
$V\subset H\subset V^{\prime }$, with continuous inclusions. The control
operator $B$ is then in $L(U,V^{\prime })$. Its adjoint is given by
\begin{equation}\label{B^*}B^{\ast
}\colon V\rightarrow U, \ \textrm{with}\  B^{\ast }v=(v(0),v).\end{equation}
 It is also useful to
note that $A^{-1}$ is well defined and that
\begin{equation}
\lbrack -A^{-1}\delta _{0}](s)=e^{\mu s},\ \text{while}\
[-A^{-1}f](s)=\int_{0}^{s}e^{-\mu (s-\sigma )}f(\sigma )d\sigma ,\ \text{for
all}\ f\in H  \label{A-1}
\end{equation}

The target functional is also interpreted on extended spaces once the
production function is described as
\begin{equation}
Q(K(\tau ))=\langle \alpha ,K(\tau )\rangle  \label{Qest}
\end{equation}%
where, the duality pairing $\langle \cdot ,\cdot \rangle $ between $V$ and $%
V^{\prime }$ has replaced the scalar product in $L^{2}$ in the original
definition (\ref{Q}) of $Q$, and $c:U\rightarrow \mathcal{U}.$ Then (\ref%
{profits}) becomes
\begin{equation*}
I(t,x;u_{0},u_{1})=\int_{t}^{+\infty }e^{-\lambda \tau }[R(\langle \alpha
,K(\tau )\rangle )-C({u}_{{0}}{(\tau ),u}_{1}{(\tau )})]d\tau
\end{equation*}

The firm's optimal investment problem falls into the wider class described
in the next theoretical sections, provided it is reformulated as a
minimization problem, where the functions $g_0$ and $h_0$ there described
are chosen as
\begin{equation}  \label{dati}
g_0(x):=-R(\langle \alpha,x\rangle), \ h_0(u_0,u_1):=C(u_0,u_1).
\end{equation}
Indeed the following Lemma holds true.

{\color{blue} }

\begin{lemma}
\label{lm:verassvintage}  Assumptions \ref{asst1} imply, along with the
above definitions of $A$ and $B$ and (\ref{dati}), that Assumptions \ref%
{asst2} are satisfied with $\omega=-\mu$. Furthermore if $p>2$, then
Assumption \ref{asst2bis} (1) is satisfied. If instead if $p>1$ and $R$ has
at most linear growth, Assumption \ref{asst2bis} (2) is satisfied.
\end{lemma}


\textit{Proof.} Assumption \ref{asst2}-(1) is satisfied with $\omega=-\mu$
since, for every $\phi\in V^{\prime }$ we have, by definition of $V^{\prime }
$
\begin{equation*}
\vert e^{\tau A}\phi\vert_{V^{\prime }} = \sup_{|f|_{V}=1} \langle e^{\tau
A}\phi,f\rangle = \sup_{|f|_{V}=1} \langle \phi,e^{\tau
A_0^*}f \rangle \le \sup_{|f|_{V}=1} \vert
\phi\vert_{V^{\prime }}\vert e^{\tau A_0^*}f\vert_{V} \le e^{-\mu t}\vert
\phi\vert_{V^{\prime }}.
\end{equation*}
Assumption \ref{asst2}-(2) is trivially satisfied as pointed out above in
the definition of $B$. By Assumption \ref{asst1}, $\alpha \in V$ and $g_0$
is a Fr\'echet differentiable convex function of $x\in V^\prime$, with
Fr\'echet differential $g^\prime_0(x)$ defined by $g^\prime_0(x)[s]=-R^%
\prime(\langle \alpha, x\rangle)\alpha (s)$ . Such differential is a
Lipschitz continuous function of $x$, with Lipschitz constant $%
Lip(g^\prime_0)= Lip(R^\prime) \vert \alpha\vert_V^2$, as
\begin{multline}
\vert g^\prime_0(x)-g^\prime_0(y)\vert_{V}\le\vert R^\prime(\langle \alpha,
x\rangle)-R^\prime(\langle \alpha, y\rangle) \vert \vert \alpha\vert_V \le
Lip(R^\prime)\vert\langle \alpha, x-y\rangle\vert\vert \alpha\vert_V \\
\le Lip(R^\prime)\vert \alpha\vert_V^2\vert x-y\vert_{V^\prime}.
\end{multline}
so that Assumption \ref{asst2}-(3) holds true.  Assumptions \ref{asst1}($ii$%
) coupled with Remark \ref{coniugate} implies both Assumptions \ref{asst2}
(4) and that $h_0^*$ is convex and Fr\'echet differentiable. The fact that $%
h_0^*$ has Lipschitz differential is implied by ($iii$), so that also (5)
holds true. Clearly ($iv $) implies (6).  {The last statement is
straightforward}. \hfill $\Box$

\begin{rem}
\label{rm:noquadraticsection3} It is important to note that, in the case
when the functions $R$ and $C$ are both quadratic, neither (1) nor (2) are
satisfied in Assumption \ref{asst2bis}. Nonetheless necessary and sufficient
conditions of optimality (see Theorem \ref{Pmp}) hold true, and the value function results regular (see Remark \ref{rm:uniquefeedback} and Section \ref{sss:RS} for details).
\qn
\end{rem}

\section{Equilibrium points}

\label{sec:mathbis}

Although the core of the section is the definition of equilibrium points of
the abstract problem (P) and the investigation of their properties, some
results are needed beforehand. Those obtained via Dynamic Programming, and
contained in \cite{FaGo2}, are recalled for the reader's convenience in {%
Section \ref{DPKNOWN}.} On the other hand, Section \ref{subsec:MP} contains
new material, and in particular a version of the Maximum Principle for
problem (P). Finally Section \ref{SS:EQPT} contains the analysis of
equilibrium points.

\subsection{Dynamic Programming for problem (P)}

\label{DPKNOWN}

We here recall the main results contained in \cite{FaGo2}. If the  value
function is defined as
\begin{equation}  \label{VF}
Z (t,x)=\inf_{u \in L^p_\lambda(t,+\infty;U)}J (t,x,u),
\end{equation}
and, if one sets $Z_0(x):=Z(0,x)$, then $Z (t,x)=e^{-\lambda t}Z_0 (x)$, so
that the Hamilton--Jacobi--Bellman equation associated to the problem by
means of Dynamic Programming reduces to that with initial time $t=0$, that
is
\begin{equation}  \label{SHJB}
-\lambda \psi(x)+\langle \psi^\prime(x)\; ,\; A
x\rangle-h_0^*(-B^*\psi^\prime(x))+g_0(x)=0,\ x\in H
\end{equation}
(with $\psi$ the unknown) whose candidate solution is $Z_0(x)$. We refer to $%
p\mapsto h_0^*(-B^*p)$ as to the \emph{Hamiltonian} function.\footnote{%
Note that the function usually called {Hamiltonian} would be $(p,x)\mapsto {%
\langle p,A x\rangle}-h_0^*(-B^*p)+g_0(x)$.}

\begin{defi}
\label{defsolSHJB} A function $\psi$ is a classical solution of the
stationary HJB equation (\ref{SHJB}) if it belongs to $\Sigma_0 (V^\prime)$
and satisfies (\ref{SHJB}) for every $x\in D(A)$.
\end{defi}

\begin{theo}
\label{th:mainnew} {Let Assumptions \ref{asst2} and \ref{asst2bis} hold.}
Then there exists a unique classical solution $\Psi$ to $(\ref{SHJB})$ and
it is given by the value function of the optimal control problem, that is
\begin{equation*}
\Psi(x)=Z_0(x)=\inf_{u\in L^p_\lambda(0,+\infty;U)}J (0,x,u).
\end{equation*}
\end{theo}

Once we have established that $\Psi$ is the unique classical solution to the
stationary HJB equation, and {since $\Psi$ is Fr\'echet differentiable with
Lipschitz derivative}, we can build optimal feedbacks and prove the
following theorem.

\begin{theo}
\label{th:uniquefeedback}  {Let Assumptions \ref{asst2} and \ref{asst2bis}
hold.} Let $t\geq 0$ and $x\in V^{\prime }$ be fixed. Then there exists a
unique optimal pair $(u^{\ast },y^{\ast })$ at $(t,x)$. The optimal state $%
y^{\ast }$ is the unique solution of the Closed Loop Equation

\begin{equation}  \label{CLE}
\begin{cases}
y^{\prime }(\tau )=Ay(\tau )+B(h_{0}^{\ast })^{\prime }(-B^{\ast }\Psi
^{\prime }(y( \tau ))), & \tau >t \\
y(t)=x\in V^{\prime }, &
\end{cases}%
\end{equation}%
while the optimal control $u^{\ast }$ is given by the feedback formula
\begin{equation*}
u^{\ast }(s)=(h_{0}^{\ast })^{\prime }(-B^{\ast }\Psi ^{\prime }(y^{\ast
}(s))).
\end{equation*}
where the optimal feedback map $x\mapsto (h_{0}^{\ast })^{\prime } (-B^{\ast
}\Psi ^{\prime}(x))$ is Lipschitz continuous.
\end{theo}

\begin{rem}
\label{rm:uniquefeedback} 
There are relevant cases when Assumption \ref{asst2bis} is not satisfied.
One such example is the case, important for the applications, when costs $g_0
$ and $h_0$ are quadratic (or linear + quadratic). Nonetheless Theorems \ref%
{th:uniquefeedback} remains true, with identical proof to that provided in
\cite{FaGo2}, if the value function $Z_0$ is in $C^1_{Lip}(V^{\prime })$.
Indeed the regularity of $Z_0$ implies that $Z_0$ is a classical solution of
the associated HJB equation \eqref{SHJB} (to this extent see e.g. \cite%
{LiYong95}, ch. 6, Proposition 1.2, p. 225). In the case of quadratic costs,
for instance, one proves that $Z_0$ is itself quadratic, and hence in $%
C^1_{Lip}(V^{\prime })$. Note also that if $Z_0$ is not differentiable, then
the closed loop equation (\ref{CLE}) holds in the weaker sense of (\ref%
{eq:weakCLE}), as specified in the next section.\qn
\end{rem}

\subsection{Maximum Principle for Problem (P)}

\label{subsec:MP}

The results contained in this section, namely Theorems \ref{Pmp} and \ref%
{pi^*}, 
are new to literature and add to the theory developed in \cite%
{Fa1,Fa2,Fa3,FaGo,FaGo2}. They establish a Maximum Principle for the problem
at hand, and connect it to the results on Dynamic Programming contained in
those papers. The reader may find all of the proofs in the Appendix, as well
as some additional results.   We advise the reader that,
differently from \cite{FaGo2} and in view of Remark \ref{rm:uniquefeedback},
the new results are proved avoiding Assumption \ref{asst2bis}. As a
consequence, if on the one hand the regularity of the value function $Z_0$
of (P) does not necessarily hold true, on the other hand we are able to
treat the case of the limit exponent $p=2$, and hence of quadratic costs $g_0
$ and $h_0$, so important for the applications.

In order to establish a maximum principle, we first need to define a dual
system associated to the mimimization problem. For all fixed $x\in V^{\prime
}$ and $t\geq 0$, we consider the equation
\begin{equation}
\pi ^{\prime }(\tau )=(\lambda -A_{1}^{\ast })\pi (\tau )-g_{0}^{\prime
}(y(\tau )),\quad \ \tau \in \lbrack t,+\infty )  \label{cost}
\end{equation}%
where $\pi :[t,+\infty )\rightarrow V$ (the dual variable, or co-state of
the system) is the unknown, and $y=y(\cdot ;t,x,u)$ is the trajectory
starting at $x$ at time $t$ and driven by control $u$, given by (\ref%
{eq:statoV'}). We assume such equation is also subject to the following
transversality condition
\begin{equation}
\lim_{T\rightarrow +\infty }e^{\left( \omega -\lambda \right) T}\pi (T)=0.
\label{tc}
\end{equation}%
When necessary, we denote any solution of (\ref{cost})(\ref{tc}) also by $%
\pi (\cdot ;t,x,u)$ or by $\pi (\cdot ;t,x)$ to remark its dependence on the
data.

Heuristically speaking, the candidate conditions of optimality associated to
the problem are the following:

\begin{equation}
\begin{cases}
y^{\prime }(\tau )=Ay(\tau )+Bu(\tau ), & \tau \geq t \\
y(t)=x &  \\
\pi ^{\prime }(\tau )=(\lambda -A_{1}^{\ast })\pi (\tau )-g_{0}^{\prime
}(y(\tau )), & \tau \geq t \\
\displaystyle\lim_{T\rightarrow +\infty }e^{(\lambda -\omega )T}\pi (T)=0, &
\\
-B^{\ast }\pi (\tau )\in \partial h_{0}(u(\tau )),\  & \tau \geq t.%
\end{cases}
\label{CO}
\end{equation}
{The ODEs for $y$ and $\pi$ appearing in \eqref{CO} are intended, as it is
usual in these cases, in \emph{mild} sense, see Definition \ref{df:mild} in
Appendix A.} Moreover, by conjugation formula, we have
\begin{equation}
-B^{\ast }\pi (\tau )\in \partial h_{0}(u(\tau ))\iff u(\tau )=
(h_{0}^{\ast})^{\prime }(-B^{\ast }\pi (\tau )).  \label{mp2}
\end{equation}%
We refer to (\ref{mp2}) as to \textit{maximum condition}. It has to be
satisfied for a.a. $\tau \geq t$.

 The conditions listed in (\ref{CO}) prove to be necessary and
sufficient for optimality for all $p\ge2$, in the sense specified next.

\begin{theo}
\label{Pmp} \emph{\textbf{(Maximum Principle).}} Let Assumptions \ref{asst2}
be satisfied. Let $p\ge2,$ $q=\frac{p}{p-1}$ and $\lambda >(2\omega)\vee
\omega $. Let $t\ge 0$, $x \in V^{\prime }$.

\begin{itemize}
\item[$(i)$] Let $(u,y)\in L^p_\lambda(t,+\infty;U)\times
L^1_{loc}(t,+\infty;V^\prime)$ be a given admissible pair at $(t,x)$. If
there exists a function $\pi\in L^q_\lambda(t,+\infty;V)$ satisfying, along
with $u$ and $y$, the system $(\ref{CO}),$ then $(u,y)$ is optimal at $(t,x)$
for the problem of minimizing $(\ref{eq:statoV'})(\ref{J in H})$.

\item[$(ii)$] Assume further that, either $p>2$ and $\lambda>0$, or $p=2$, $%
\lambda>0$ and $\omega<0$. Then the viceversa of (i) holds, i.e., any couple
$(u^*,y^*)$ optimal at $(t,x)$ necessarily admits a costate $\pi\in
L^q_\lambda(t,+\infty;V)$ satisfying, along with $u^*$ and $y^*$, system $(%
\ref{CO}).$ %
\end{itemize}
\end{theo}


The next theorem containes the so-called \emph{co-state inclusion}. Note
that the case of $p=2$ is discussed separately, as the value function is not
necessarily Fr\'echet differentiable (unless Assumption \ref{asst2bis} holds
or ad hoc regularity results are given).

\begin{theo}
\label{pi^*} \emph{\textbf{(Co-state inclusion).}} In Assumptions \ref{asst2}%
, for $\lambda>\max\{0,\omega,2\omega\}$, suppose that either $p>2$, or $p=2$
and $\omega<0$. Let $(u^*,y^*)$ be optimal at $(t,x)\in [0,+\infty)\times
V^{\prime }$, and let $\pi^*(\cdot;t,x)\in L^q_\lambda(t,+\infty;V)$ be the
associated co-state. Let also $Z_0$ be the value function of problem (P).
Then
\begin{equation*}
\pi^*(\tau;t,x)=\pi^*(\tau;\tau,y^*(\tau)) \in \partial Z_0(y^*(\tau)),
\qquad \forall \tau\ge t.
\end{equation*}
where $\partial Z_0$ is the subdifferential of the convex function $Z_0$. If
in addition $p>2$, then $Z_0\in\Sigma_0(V^{\prime })$ and $Z_0$ coincides
with $\Psi$, so that
\begin{equation*}
\pi^*(\tau;t,x)=\pi^*(\tau;\tau,y^*(\tau))=\Psi^\prime(y^*(\tau)), \qquad
\forall \tau\ge t.
\end{equation*}
\end{theo}

\begin{rem}
\label{rm:weakCLE} Note that, for $p\ge2$ and $\lambda>0$, and by making use
of Theorem \ref{Pmp}, and of equations \eqref{CO} and \eqref{mp2}, one
obtains
\begin{equation*}
y^{\prime }(\tau )=Ay(\tau ) +B(h_{0}^{\ast})^{\prime }(-B^{\ast }\pi (\tau
)), \qquad \tau \geq t
\end{equation*}
so that the general version of the closed loop equation \eqref{CLE} becomes
a differential inclusion
\begin{equation}  \label{eq:weakCLE}
y^{\prime }(\tau )\in Ay(\tau ) +B(h_{0}^{\ast})^{\prime }(-B^{\ast
}\partial Z_0(y(\tau))), \qquad \tau \geq t.
\end{equation}
also to be intended in mild sense.\qn
\end{rem}

\color{black}

\subsection{Equilibrium points}

\label{SS:EQPT} We give two different definitions of equilibrium points for
problems (P), and later show to which extent they are equivalent.


\begin{defi}
\label{def:eq} A \emph{MP-equilibrium point} of problem $(P)$ is any
stationary solution $\left( x,\pi,u \right) \in V^{\prime }\times V\times U$
of $(\ref{CO})$. This is equivalent to require that $\left( x,\pi,u \right)$
belongs to $D(A)\times D(A_1^*)\times U$ and satisfies
\begin{equation}  \label{def:MPEP}
\begin{cases}
Ax+Bu=0 \\
(\lambda -A_{1}^{\ast })\pi -g_{0}^{\prime }(x)=0, \\
u=(h_{0}^{\ast })^{\prime }(-B^{\ast }\pi )).%
\end{cases}%
\end{equation}
 A \emph{CLE-equilibrium point} of problem $(P)$ is any $x\in
V^{\prime }$ that is a stationary solution of the closed loop equation
\eqref{eq:weakCLE}. This is equivalent to require $x\in D(A)$ and
\begin{equation}  \label{equili1newweak}
Ax+B(h_{0}^{\ast })^{\prime }(-B^{\ast }\partial Z_0 (x))\ni 0.
\end{equation}%

\end{defi}

\begin{rem}
When $0\in \rho(A)$ and $\lambda\in \rho(A_1^*)$, then (\ref{def:MPEP}) is
equivalent to
\begin{equation}  \label{eq:PMP-1}
\begin{cases}
x=-A^{-1}Bu \\
\pi =(\lambda -A_{1}^{\ast })^{-1}g_{0}^{\prime }(x), \\
u=(h_{0}^{\ast })^{\prime }(-B^{\ast }\pi )).%
\end{cases}%
\end{equation}
and 
 (\ref{equili1newweak}) is equivalent to
\begin{equation}  \label{equili1weak}
x\in-A^{-1}B(h_{0}^{\ast })^{\prime }(-B^{\ast }\partial Z_0(x)).
\end{equation}
 As a consequence of Remark \ref{rm:uniquefeedback}, the equations
\eqref{eq:weakCLE}, \eqref{equili1newweak} and \eqref{equili1weak} hold as
equalities with $\Psi^\prime(x)$ in place of $\partial Z_0(x)$ when $Z_0$
Fr\'echet differentiable in $V^{\prime }$ (e.g. when $p>2$, or when regularity can be proven separately).\qn
\end{rem}

 The proof of the equivalences in the above definition is
straightforward as they are based on standard regularity of convolutions of
semigroups. We omit them for brevity.

\bigskip

We have the following result.

{\color{blue} }

\begin{theo}
\label{th:equiv} Let Assumptions \ref{asst2} be satisfied, $p\ge2$, $%
\lambda>(2\omega)\vee\omega$.

\begin{itemize}
\item[$(i)$] Let $(\bar x,\bar \pi, \bar u)\in D(A) \times D(A_1^*)\times U$
be any MP-equilibrium point. Then the constant control $\bar u$ is optimal
at $(0,\bar x)$ and
\begin{equation}  \label{eqlts}
A\bar x+B(h_0^*)^\prime(-B^*(\lambda-A_1^*)^{-1}g_0^\prime(\bar x))=0,
\end{equation}
moreover $\bar x$ is a CLE-equilibrium point and
\begin{equation}  \label{eq:MPisCLEweak}
\partial Z_0 (\bar x)\ni (\lambda-A_1^*)^{-1}g_0^\prime(\bar x).
\end{equation}

\item[$(ii)$] Let $\hat x\in D(A)$ be a CLE-equilibrium point, $\lambda>0$.
Let either $p>2$, or $p=2$ and $\omega<0$. Assume that $Z_0 $ is Fr\'echet
differentiable in $V^{\prime }$. Then $(\hat x,\hat \pi, \hat u)$, where
\begin{equation*}
\hat \pi:=(\lambda-A_1^*)^{-1}g_0^\prime(\hat x) \quad and \quad \hat
u:=(h_{0}^{\ast })^{\prime }(-B^{\ast }\hat{\pi}),
\end{equation*}
is an MP-equilibrium point, the control $\hat u$ is optimal at $(0,{\color{blue}\hat x})$ and $%
Z_0^\prime(\hat x)=\hat \pi = (\lambda-A_1^*)^{-1}g_0^\prime(\hat x)$.
\end{itemize}
\end{theo}


\medskip

One important consequence of the above theorem is that it provides the
following equation {for} a CLE-equilibrium point (or for the first component
of an MP-equilibrium point)
\begin{equation}  \label{eq:eqptx}
A x+B(h_0^*)^\prime(-B^*(\lambda-A_1^*)^{-1}g_0^\prime(x))=0.
\end{equation}
In addition, whenever $0 \in \rho(A)$ (this assumption is satisfied in the
optimal investment problem with vintage capital described in Section 2)
solutions of \eqref{eq:eqptx} can be regarded as fixed points of the
operator $T:V^\prime\to V^\prime$, defined by
\begin{equation}  \label{Tlambda}
Tx:=-A^{-1}B(h_0^*)^\prime(-B^*(\lambda-A_1^*)^{-1}g_0^\prime( x)).
\end{equation}
For the applications, the most efficient way of making use of such relations
is to rewrite them in terms of the specific sets of data, and compute when
possible the optimal equilibrium {distributions}. In particular, in Section %
\ref{sec:avc} we will see how (\ref{Tlambda}) is interpreted in terms of the
data of optimal investment with vintage capital, so that fixed points of $T$
may be directly computed by solving a numeric equation.

However, in the general case, it is possible to provide sufficient
conditions for the existence and uniqueness of a fixed point of the operator
$T$ using well known fixed point theorems although, as one expects, such
conditions may hardly be very sharp. To this extent, we provide here only
Lemma \ref{fixedpoint}, which is a straightforward application of the
contraction mapping principle.

\begin{lemma}
\label{fixedpoint} Let Assumptions of Theorem \ref{th:equiv} be satisfied.
Assume moreover that
\begin{equation*}
{\lambda-\omega}>\Vert (A)^{-1}\Vert_{\mathcal L(V^{\prime })}\Vert
B\Vert^2_{\mathcal L(U,V^\prime)} [(h_0^*)^\prime][g_0^\prime].
\end{equation*}
Then there exists a unique solution $\bar x\in D(A)$ to the equation %
\eqref{Tlambda}.
\end{lemma}

\begin{rem}
The operator $T$ above is considered as an operator from $V^{\prime }$ to
itself. Since its image is contained in $D(A)$, when looking for fixed
points, it is also equivalent to look at it as an operator from $D(A)$ to
itself, considered as a subspace of $V'$, as
done in Lemma \ref{lm:fxptvintage}.\qn
\end{rem}

\begin{rem}
{All above results} could be generalized to the case in which we have state
constraints and the function $g_0$ is   convex  but not
necessarily Fr\'echet differentiable. This could be done using the results
of \cite{FaSC} and generalizing them to the infinite horizon case, using the
same arguments in \cite{G3}. Clearly, at points where $g_0$ is not Fr\'echet
differentiable, one would have to choose an element of the
subdifferential  of $g_0$.\qn
\end{rem}

\subsection{Stability}

\label{SS:STABILITY} Once existence (and possibly uniqueness) of equilibrium
points is proven, it is possible to study their stability properties
adapting known results such as those in \cite{Kato95} or in chapter 9 in
\cite{Lunardibook}, or by direct proof, as we see next.
In all cases, stability will be proven with respect to the topology of $V'$. For the reader's convenience we recall the definition here below.

\begin{defi}\label{def:stab}
A CLE-equilibrim point $\bar x\in D(A)$ is \emph{stable} in the
topology of $V^{\prime }$ if, $\forall\epsilon >0, \exists\delta>0$
such that, if $x\in V^{\prime }$ and $x^*(\cdot)$ is the optimal
trajectory starting at $x$, then $|x-\bar x|_{V^{\prime
}}< \delta\Rightarrow|x^*(t)-\bar x|_{V^{\prime }}< \epsilon$. If in addition $ \lim_{t\to\infty}\vert x^*(t)-\bar x\vert_{V'}=0$ then $\bar x$ is \emph{asymptotically} stable. Finally, if the same property hold true for all $x\in V'$, then $\bar x$ is
\emph{globally}  asymptotically stable.
\end{defi}

The first criterium to establish stability is contained in the following proposition and makes use of
 the linearization method. The proof follows from Corollary 2.2 in \cite{Kato95}.

\begin{prop}
\emph{(Stability by linearization)} \label{pr:stability}  {Let Assumption %
\ref{asst2} be satisfied and $\Psi\in C^1_{Lip}$. For $x \in V^\prime$ set $%
f(x):=B(h_0^*)^\prime(-B^*\Psi^\prime(x))$,} and assume that $\bar x \in D(A)
$ is a CLE-equilibrium point for (P), that $f$ is continuously Fr\'echet
differentiable at a neighborhood of $\bar x$, and denote by $\sigma_{\bar x}$
the spectrum of the operator $A+f^\prime(\bar x)$. If $\sup (Re \sigma_{\bar
x}) <0$, then $\bar x$ is stable in the topology of $V^{\prime }$. Moreover,
it is also asymptotically stable in the topology of $V^{\prime }$. If $\sup
(Re \sigma_{\bar x})>0$, $\bar x$ is unstable in the topology of $V^{\prime }
$.
\end{prop}

Another result that can be used is the following. Recall that $%
\langle\cdot,\cdot\rangle_{V^{\prime }}$ indicates the inner product in $%
V^{\prime }$.

\begin{prop}
\emph{(Stability by dissipativity)} \label{pr:stabilitybis} {Let Assumption %
\ref{asst2} be satisfied and let $\Psi \in C_{Lip}^{1}(V^{\prime })$. For $%
x\in V^{\prime }$, and $f(x)=B(h_{0}^{\ast })^{\prime }(-B^{\ast }\Psi
_{0}^{\prime }(x))$,} and assume that, when $x\in D(A)$, the solution of the
closed loop equation
\begin{equation}
\begin{cases}
y^{\prime }(t)=Ay(t)+f(y(t)), & t>0 \\
y(0)=x, &
\end{cases}
\label{CLEf}
\end{equation}%
belongs to $D(A)$ for all $t\geq 0$. Let $\bar{x}\in D(A)$ be a
CLE-equilibrium point for (P). Assume that $A+f$ is dissipative {near $\bar{x%
}$}, i.e. there exists an open ball $I(\bar{x})$ in $V^{\prime }$ centered
at $\bar{x}$, and  $\xi\le 0$
  such that, for every $x\in I(\bar{x})\cap D(A)$,
\begin{equation}
\big(A(x-\bar{x})+f(x)-f(\bar{x})\big\vert x-\bar{x}\big)_{V^{\prime }}\leq
 {\xi} |x-\bar{x}|_{V^{\prime }}^{2}.  \label{eq:AdissV'}
\end{equation}%
Then $\bar{x}$ is stable in the topology of $V^{\prime }$. If $\xi <0$ then
$\bar{x}$ is asymptotically stable in the topology of $V^{\prime }$. If $A+f$
is dissipative on the whole $V^{\prime }$ and $\xi <0$ then $\bar{x}$ is
globally asymptotically stable.
\end{prop}

\begin{cor}
\label{cor:Adiss} Let the assumptions of Proposition $\ref{pr:stabilitybis}$
be verified, except \eqref{eq:AdissV'}. Let the operator $A$ satisfy
$( Ax\vert x)_{V^{\prime }}\le -\theta |x|_{V^{\prime }}^2$, for all $x$ in $%
V^{\prime }$, with $\theta>0$ a fixed constant. If there exists a
neighborhood $I$ of $\bar x$ where $f$ is Lipschitz continuous (in the
topology of $V^{\prime }$) with Lipschitz constant strictly smaller than $%
\theta$, then $\bar x$ is asymptotically stable in the topology of $%
V^{\prime }$. If $f$ is Lipschitz continuous in $V^{\prime }$ with Lipschitz
constant strictly smaller than $\theta$ then $\bar x$ is globally
asymptotically stable.
\end{cor}

In particular, the above corollary may be applied to the examples  in
Section 5, see e.g. subsection \ref{sss:RS}. 

\section{Application to Optimal Investment with Vintage Capital}

\label{sec:avc}

The aim of this section is to show how valuable our general theory can be
when analyzing specific applications, and in particular when trying to
derive analytic formulas for {equilibrium distributions}. This process
unfolds by computing {MP-equilibrium/CLE-equilibrium} points for that
problem rephrased in abstract form as in Section \ref{INTROSUBMAT}. 


\smallskip

We begin by noting that the value function $V(t,x)=e^{-\lambda
t}V(0,x)=e^{-\lambda t}V_0(x)$ of the optimal control problem described in
Section \ref{MODEL} (see \eqref{eq:VFsec2}), satisfies, as a consequence of (%
\ref{profits}), (\ref{dati}) and (\ref{VF}),
\begin{equation*}
V_0(x)=-Z_0(x)
\end{equation*}
where $Z_0$ is the value function, defined in
Section \ref{DPKNOWN},  of  the abstract problem (P)  for $t=0$. Note that $V_0$ is a concave function, as
$Z_0$ is convex. Note also that, under additional assumptions (e.g.
Assumption \ref{asst2bis}, or regularity assumptions on $Z_0 $), $Z_0 $ is the unique classical solution of
HJB equation (\ref{SHJB}) in the sense of Definition \ref{defsolSHJB}. As a
consequence, the natural co-state for the maximization problem would be $$
\zeta(\tau,s)=-\pi(\tau)[s],$$ with $\pi$ the co-state of the abstract problem whose properties are
described in Theorem \ref{Pmp} and \ref{pi^*}. Then the optimality
conditions (\ref{CO}) for a triplet $(K^*,\zeta^*, u^*)$ can be written as the following set of equations
\begin{equation}
u_{0}^{\ast }(\tau )=(C_{0}^{\ast })^{\prime }(\zeta^{\ast }(\tau ,0)),\ \ \
u_{1}^{\ast }(\tau ,s)=\big( (C_{1}^{\ast })^{\prime }(\zeta ^{\ast }(\tau
,\cdot ))\big) [s]=[{c_{1}}(s,\cdot )^{\ast }]^{\prime }(\zeta ^{\ast
}(\tau ,s))  \label{u^*}
\end{equation}%
\begin{equation}  \label{p^*}
\zeta ^{\ast }(\tau ,s)=\int_{s}^{\overline{s}}e^{-(\lambda +\mu )(\xi
-s)}\,R^{\prime }\left( \int_{0}^{\overline{s}}\alpha (\theta )K^{\ast
}(\tau +\xi -s,\theta )d\theta \right) \,\alpha (\xi )d\xi ;
\end{equation}
\begin{equation} \displaystyle\lim_{T\rightarrow +\infty }e^{(\lambda -\omega )T}\zeta (T,s)=0, \ a.a.\ s\in[0,\bar s]. \end{equation}
\begin{equation}
K^{\ast }(\tau ,s)=%
\begin{cases}
e^{-\mu (\tau -s)}x(s-\tau +t)+\int_{0}^{\tau -t}e^{-\mu \sigma }u_{1}^{\ast
}(\tau -\sigma ,s-\sigma )d\sigma & s\in \lbrack \tau -t,\overline{s}],\text{
}\tau \in \lbrack t,\overline{s}+t] \\
e^{-\mu (\tau -s)}u_{0}^{\ast }(\tau -s)+\int_{0}^{s}e^{-\mu \sigma
}u_{1}^{\ast }(\tau -\sigma ,s-\sigma )d\sigma & s\in \lbrack 0,\tau -t],%
\text{ }\tau \in \lbrack t,\overline{s}+t] \\
0 & s\in \lbrack 0,\overline{s}],\text{ \ }\tau \in (\overline{s}+t,+\infty )
\end{cases}
\label{K^*}
\end{equation}

Note that (\ref{u^*}) and (\ref{p^*}) are derived from (\ref{CO}) by making
use of (\ref{mp2}) and (\ref{addi}), while (\ref{K^*}) is well known and can
be obtained by means of characteristics method (see e.g.  \cite{BG2}).

The following proposition is an immediate consequence of Theorem \ref{Pmp} and of
Lemma \ref{lm:verassvintage}.

\begin{prop}
\label{lm:vintagepmp} Under Assumptions \ref{asst1}, with $p\ge 2$ and $%
\lambda >0$, the optimality conditions (\ref{u^*})(\ref{p^*})(\ref{K^*}) are
necessary and sufficient for a couple $(u^*, K^*)$ to be optimal at $x$ for
the problem of optimal investment with vintage capital described in Section %
\ref{MODEL}. 
\end{prop}

\subsection{Characterization of Equilibrium Points}
It is  natural to define an equilibrium point for the problem  consistently with Section \ref{SS:EQPT}. For the reader's
convenience,  Definition \ref{def:eq}  is  reformulated below in terms of the problem of optimal investment with vintage capital.

It is important to note that an equilibrium point is actually a function of the variable $s$ (although independent of $t$) dependent on the variable $s$, hence an \emph{%
equilibrium distribution}.

\begin{defi}\label{defeqapp} In reference to the  the optimal investment  {problem}
with vintage capital:
\begin{itemize}
\item[$(i)$] a \emph{MP-equilibrium
point} is  a stationary solution $({x}, \zeta,(u_0,u_1) )\in L^2{(0,\bar s)} \times
H^2(0,\bar s)\times (\mathbb{R}\times L^2{(0,\bar s)})$ of the system of
equations $(\ref{u^*})(\ref{p^*})(\ref{K^*})$;
\item[$(ii)$] for $V_0$   Fr\'echet differentiable, a CLE-equilibrium point  is any ${x}\in L^2{(0,\bar s)}$ which is a stationary
solution of  equation \eqref{K^*} when
\begin{equation*}
u_{0}^{\ast }(\tau )= [(C_{0}^{\ast })^{\prime }(V'_0(K^*(\tau,\cdot))][0],\ \ \
u_{1}^{\ast }(\tau ,s)= \left[ (C_{1}^{\ast})^{\prime }( V'_0(K^*(\tau,\cdot)))\right] [s]. \end{equation*}\end{itemize}
\end{defi}

\begin{rem} Several remarks are here due.
 \begin{enumerate}
\item Theorem \ref{th:equiv} implies, in the assumptions of Proposition \ref{lm:vintagepmp} that $(i)$ and $(ii)$ are equivalent in the following sense:
 the first component $x$ of a MP-equilibrium point $(x,p,(u_0,u_1))$ is also a
CLE-equilibrium point; conversely, when  $V_0$ is Fr\'echet differentiable, a CLE-equilibrium point $x$ can be used to build  a MP-equilibrium point  by means of  $(\ref{u^*})(\ref{p^*})(\ref{K^*})$, having $x$ as   first component.
\item  If  $V_0$ is  not Fr\'echet differentiable, the closed loop equation (as well as the definition above) may be generalized  to a differential inclusion in the sense of \eqref{eq:weakCLE}, where $V_0^{\prime }$ is replaced by  the superdifferential $\partial V_0$.
\item  Definition \ref{defeqapp} is consistent with Definition \ref{def:eq} as here $D(A)=L^2{(0,\bar s)}$ and $%
D(A_1^*)\subseteq H^2{(0,\bar s)}$.\color{blue}
\end{enumerate}\qn
\end{rem}

We further characterize {MP-equilibrium}/CLE-equilibrium points as fixed
points of a suitable operator. To this extent we define
\begin{equation}  \label{alphabar}
\bar{\alpha}(s)=\int_{s}^{\bar{s}%
}e^{-(\mu +\lambda )(\sigma -s)}\alpha (\sigma )d\sigma
\end{equation}%
\textit{i.e.} $\bar{\alpha}(s)$ is the \textit{discounted return} associated
with a unit of capital of vintage $s$.

\begin{lemma}
\label{lm:fxptvintage} Under Assumptions \ref{asst1}, $\bar x$ is a CLE-equilibrium point if and only if it is a
 fixed point of the operator $T:L^{2}{(0,\bar{s})%
}\rightarrow L^{2}{(0,\bar{s})}$ defined by
\begin{equation}
(Tx)[s]=\left( C_{0}^{\ast }\right) ^{\prime }\left( R^{\prime
}(\left\langle \alpha ,x\right\rangle )\bar{\alpha}(0)\right) e^{-\mu
s}+\int_{0}^{s}e^{-\mu (s-\sigma )}[{c_{1}}(\sigma ,\cdot )^{\ast }]^{\prime
}\left( R^{\prime }(\left\langle \alpha ,x\right\rangle )\bar{\alpha}(\sigma
)\right) d\sigma,  \label{Texpl}
\end{equation}
that is, if and only if $(T\bar x)[s]=\bar x(s)$ for {a.e.} $s$ in $[0,\bar{%
s}]$. Moreover $(\bar x,\bar \zeta,(\bar
u_0,\bar u_1))$ where
\begin{equation}  \label{eq:pu}
\bar \zeta(s)=R^\prime (\langle \alpha, \bar x\rangle)\bar\alpha(s), \ \
\bar u(s) =(h_0^*)^\prime( {B^*\bar \zeta})(s), \ \text{for\ a.e.\ }s\in[0,\bar{%
s}]
\end{equation}
is  a MP-equilibrium point.
\end{lemma}

The proof of the lemma is contained in Appendix \ref{AC}.

\smallskip

Note that solving the equation $Tx=x$ within a space of functions is not particularly handy. Nonetheless solving such functional equation is equivalent - and   in the generality of cases -  to solving a \emph{numeric equation}.
In this sense, the following theorem contains the most interesting result of the section.

\begin{theo}
\label{equgen} Let Assumptions \ref{asst1} be satisfied, and let $T$ be given
by \eqref{Texpl}. Moreover,   for any $\eta\in\mathbb R$ and $s\in[0,\bar s]$,  consider the function
\begin{equation}  \label{eq:defFeta}
F(\eta )[s]=\left( C_{0}^{\ast }\right) ^{\prime }\left( \eta \bar{\alpha}%
(0)\right) e^{-\mu s}+\int_{0}^{s}e^{-\mu (s-\sigma )}[{c_{1}}(\sigma ,\cdot
)^{\ast }]^{\prime }\left( \eta \bar{\alpha}(\sigma )\right) d\sigma .
\end{equation}
Then ${\bar x}\in L^{2}(0,\bar{s})$ is a solution of $Tx=x$, if and only if
\begin{equation}
{\bar x}(s)=F(\bar{\eta })[s]  \label{punteqgen}
\end{equation}%
with $\bar{\eta }$   a solution  in $\mathbb R$  of
\begin{equation}
\eta =R^{\prime }(\left\langle \alpha ,F(\eta )\right\rangle ).
\label{eta=R}
\end{equation}
\end{theo}

The proof of the theorem is contained in Appendix \ref{AC}.

\begin{rem}\label{rm:scuni}
Note that the solution of \eqref{eta=R} is unique and
nonnegative when, for instance,   $R^{\prime }(\left\langle \alpha ,F(0 )\right\rangle )\ge0.$
Indeed  $R^\prime$ is decreasing and $C_{0}^{\ast }$ and ${c_{1}}%
(\sigma ,\cdot )^{\ast }$ are convex functions, then the right hand side of (%
\ref{eta=R}) is a positive decreasing function of $\eta$. Hence the function
\begin{equation}  \label{eq:theta}
\theta(\eta):= \eta -R^{\prime }(\left\langle \alpha ,F(\eta )\right\rangle )
\end{equation}
satisfies $\theta(0)\le0$, is strictly increasing to $+\infty$, and hence has exactly one
nonnegative zero.\qn
\end{rem}

\medskip

Note that (\ref{Texpl}) and (\ref{punteqgen}) are general formulas, holding
for any choice of costs and revenues, as long as they satisfy Assumptions %
\ref{asst1}. More explicit formulas for the optimal equilibrium distribution  may be
derived once costs and revenues are further specified.

\subsection{Linear-quadratic costs}

\label{sec:lqc} Now we make formulas more explicit in the case of cost
functions satisfying $(\ref{lqcost})$. We   derive
\begin{equation}
C_{0}^{\ast }(p_{0})=\frac{(p_{0}-q_{0})^{2}}{4\beta _{0}},\text{ \ \ }%
C_{1}^{\ast }(p_{1})=\int_{0}^{\bar{s}}\frac{(p_{1}(s)-q_{1}(s))^{2}}{4\beta
_{1}(s)}ds  \label{ipo3}
\end{equation}%
so that
\begin{equation*}
\left( C_{0}^{\ast }\right) ^{\prime }(p_{0})=\frac{p_{0}-q_{0}}{2\beta _{0}}%
,\text{ \ \ }\left( C_{1}^{\ast }\right) ^{\prime }(p_{1})(s)=\frac{%
p_{1}(s)-q_{1}(s)}{2\beta _{1}(s)}
\end{equation*}%
In this case, (\ref{Texpl}) becomes
\begin{eqnarray}
(Tx)[s] &=&\frac{R^{\prime }(\left\langle \alpha ,x\right\rangle )\overline{%
\alpha }(0)-q_{0}}{2\beta _{0}}e^{-\mu s}+\int_{0}^{s}e^{-\mu (s-\sigma )}%
\frac{R^{\prime }(\left\langle \alpha ,x\right\rangle )\overline{\alpha }%
(\sigma )-q_{1}(\sigma )}{2\beta _{1}(\sigma )}d\sigma  \notag
\label{Texpl1} \\
&=&R^{\prime }(\left\langle \alpha ,x\right\rangle )w_{1}(s)-w_{2}(s)
\end{eqnarray}%
where $w_{1}$ and $w_{2}$ are the positive functions
\begin{equation}
w_{1}(s)=\frac{\bar{\alpha}(0)}{2\beta _{0}}e^{-\mu s}+\int_{0}^{s}e^{-\mu
(s-\sigma )}\frac{\bar{\alpha}(\sigma )}{2\beta _{1}(\sigma )}d\sigma
\label{w_1}
\end{equation}%
\begin{equation}
w_{2}(s)=\frac{q_{0}}{2\beta _{0}}e^{-\mu s}+\int_{0}^{s}e^{-\mu (s-\sigma )}%
\frac{q_{1}(\sigma )}{2\beta _{1}(\sigma )}d\sigma .  \label{w_2}
\end{equation}%
If in addition we define the positive coefficients
\begin{equation}
c_{1}=\langle w_{1},\alpha \rangle =\int_{0}^{\bar{s}}\alpha (s)w_{1}(s)ds,\
\ c_{2}=\langle w_{2},\alpha \rangle =\int_{0}^{\bar{s}}\alpha (s)w_{2}(s)ds.
\label{c12}
\end{equation}%
then the following result follows as a consequence of Theorem \ref{equgen}.

\begin{cor}
\label{equil1} Let Assumption \ref{asst1} and \eqref{lqcost}  be satisfied.
Let $w_{1},$ $w_{2}$, $c_{1},$ and $c_{2}$ be defined respectively by \eqref{w_1}, \eqref{w_2}  and \eqref{c12}.
Then $(\bar x, \bar \zeta, \bar u)$ is a MP-equilibrium point
if and only if $\eta \in  \mathbb{R}$ is  a solution of
\begin{equation}
\eta =R^{\prime }(\eta c_{1}-c_{2}).  \label{eq:eta}
\end{equation}
and
\begin{equation*}
\bar x(s)=-w_{2}(s)+\eta w_{1}(s).
\end{equation*}
 and moreover $\bar \zeta$ and $\bar u$ are given by \eqref{eq:pu}.
The constant control $\bar u$ is optimal at ${\color{red}\bar x}$, 
  $\bar x$ is also a CLE-equilibrium point and,
if $V_0$ is Fr\'echet differentiable, then it is also the unique CLE-equilibrium point. \color{black}
\end{cor}

\begin{rem}
If the CLE-equilibrium point $\bar x$ identified by Corollary \ref{equil1}
is such that $\bar x(s)\ge0$ at all $s$, then $\bar x$ is also a
CLE-equilibrium point for the problem with state constraints $k(\tau,s)\ge0$
for all $s$ and $\tau$ (see also Remarks \ref{sc1} and \ref{sc2}).\qn
\end{rem}

\begin{rem}
Note that Assumptions \ref{asst1} are satisfied here with $p=2$, so that $V_0$
 is not necessarily Fr\'echet differentiable. That implies that, although the first component $\bar x$ of a
MP-equilibrium point $(\bar x,\bar \zeta, \bar u)$ is also a CLE-equilibrium point, the viceversa may fail:  there may be  CLE-equilibrium points which do not derive as first components
of a MP- equilibrium point, i.e. solutions of the stationary closed loop equation which  fail to be optimal. For a further discussion on regularity of $V_0$, the reader is referrred to Section \ref{sss:RS}.\qn
\end{rem}

Once $R$ is chosen, the results in Corollary \ref{equil1} leads to an
explicit formula for that CLE-equilibrium point, as illustrated in the next
lemma.

\begin{lemma}
\label{theo:eq1} In the assumptions of Corollary \ref{equil1}, there exists a unique
 CLE-equilibrium point $\bar {x}$, described by the formuals below, for the associated choices of the revenue $R$:
\begin{enumerate}
\item[(i)] If $R(Q)=-aQ^2+bQ$, then
\begin{equation*}
\bar x=-w_2-\frac{2ac_2+b}{1+2ac_1} w_1;
\end{equation*}

\item[(ii)] If $R(Q)=\ln(1+Q),$ for $Q\ge0$ and $R(Q)=Q$ for $Q<0$, then
\begin{equation*}
\bar x=-w_2+\frac{\sqrt{(1-c_2)^2+4c_1}-(1-c_2)}{2c_1}\; w_1
\end{equation*}

\item[(iii)] If $R(Q)=b[(\nu+Q)^\gamma-\nu],$ with $\gamma\in(0,1)$, $%
b,\nu>0 $, for $Q\ge0$ and $R(Q)=\gamma\nu^{\gamma -1} Q$ for $Q<0$, then $%
\bar x=-w_2+\bar \eta w_1$ where $\bar \eta$ is the unique positive solution of
\begin{equation*}
\eta=\frac{b\gamma}{(\nu+c_1\eta-c_2)^{1-\gamma}}.
\end{equation*}

\item[(iv)] If $R(Q)=bQ^{\gamma },$ with $\gamma \in (0,1)$, $b>0$, for $%
Q\geq 0,$ and $R(Q)=-\infty $ for $Q<0$ (case with state constraints) then ${%
\bar x}=-w_{2}+\bar{\eta}w_{1}$ where $\bar{\eta}$ is the unique positive
solution of
\begin{equation*}
\eta =\frac{b\gamma }{(c_{1}\eta -c_{2})^{1-\gamma }}.
\end{equation*}
\end{enumerate}
\end{lemma}

\begin{proof}
The proof {follows from straightforward computations.}
\end{proof}

\subsubsection{Stability of Equilibrium Distributions} \label{sss:RS}
We close the section on linear-quadratic costs \eqref{lqcost} by briefly discussing  stability of equilibrium distributions and, in some subcases, the regularity of the value function,   by applying the results contained in Section \ref{SS:STABILITY}.
The concept of stability here used  is that of Definition \ref{def:stab}, 
which is natural in this context. We remark though that the  convergence of functions there mentioned (i.e. in the topology of $V'$)  is \emph{not} a convergence  in the space $L^2(0,\bar s)$ but, roughly speaking,  the (weaker) convergence of their primitive functions.   
\smallskip
 
\begin{lemma}\label{stablin}
In the assumptions of Corollary \ref{equil1}, suppose in addition that the value function $V_0$ is Fr\'echet differentiable, and set $\xi=-\mu+  {[V_0]_L \vert \delta
_{0}\vert_{V^{\prime }}^2} /({4\beta_0})$, where $[V_0]_L$ indicates the Lipschitz constant of the gradient $V_0'$.
If $\xi\le0$ (respectively, $\xi<0$) then  $\bar x$ is stable (resp.,  asimptotically stable) 
in the sense of  Definition \ref{def:stab}.
\end{lemma}
 
The proof of the lemma is contained in Appendix \ref{AC}.

\begin{rem}\label{stabqc}
In particular, the previous Lemma applies when $R$ is of the type
described in Lemma \ref{theo:eq1} $(i)$. Indeed with some extra work   one shows that in this case
the value function of the abstract problem is of type $$%
\Psi(x)=\langle Cx,x\rangle+\langle d, x\rangle+e,$$ for a suitable linear
operator $C:V^\prime\to V$, $d\in V$ and $e\in\mathbb{R}$, where $V$ and $V'$ are the spaces introduced in Subsection \ref{INTROSUBMAT}. Hence $\Psi$ is differentiable with Fr\'echet differential
$\Psi^\prime(x)=Cx+d$, and $[\Psi]_L=\Vert C\Vert_{\mathcal L(V',V)}$  (for a proof, we refer the reader to  \cite{LT}, vol 1, ch.2).   This applies in particular to the linear-quadratic examples of Section \ref{sens}.
\qn
 \end{rem}

\subsection{Linear-quadratic costs, constrained control}

We now choose costs as in $(\ref{lqc})$. We then derive
\begin{equation*}
g_{\beta ,M}^{\ast }(v)=\sup_{|u|\leq M}\left\{ vu-\beta u^{2}\right\}
=\left\{
\begin{array}{c}
\frac{v^{2}}{4\beta } \\
M|v|-\beta M^{2}%
\end{array}%
\begin{array}{c}
|v|\leq 2\beta \ M \\
otherwise\
\end{array}%
\right.
\end{equation*}%
Note that $g_{\beta ,M}^{\ast }$ is a $C^{1}$ function, with Lipschitz
derivative%
\begin{equation*}
\left( g_{\beta ,M}^{\ast }\right) ^{\prime }(v)=\left\{
\begin{array}{c}
\frac{v}{2\beta } \\
M \\
-M%
\end{array}%
\begin{array}{c}
|v|\leq 2\beta \ M \\
v>2\beta \ M \\
v<2\beta \ M\
\end{array}%
\right.
\end{equation*}%
As a consequence, the Legendre transform of $C$ is
\begin{equation*}
C^{\ast }(v)=g_{\beta _{0},M_{0}}^{\ast }(\beta _{0}-q_{0})+\int_{0}^{\bar{s}%
}g_{\beta _{1}(s),M_{1}}^{\ast }(v_{1}(s)-q_{1}(s))ds
\end{equation*}%
which is Fr\'echet differentiable with differential
\begin{equation*}
\left( C^{\ast }\right) ^{\prime }(v)(s)=\left( \left( g_{\beta
_{0},M_{0}}^{\ast }\right) ^{\prime }(v_{0}-q_{0});\left( g_{\beta
_{1}(s),M_{1}}^{\ast }\right) ^{\prime }(v_{1}(s)-q_{1}(s))\right)
\end{equation*}

while the operator (\ref{Texpl}) is given by
\begin{equation*}
T^{M}x(s)=\left( g_{\beta _{0},M_{0}}^{\ast }\right) ^{\prime }(R^{\prime
}(\left\langle \alpha ,x\right\rangle )\overline{\alpha }(0)-q_{0})e^{-\mu
s}+\int_{0}^{s}e^{-\mu (s-\sigma )}\left( g_{\beta _{1}(\sigma
),M_{1}}^{\ast }\right) ^{\prime }(R^{\prime }(\left\langle \alpha
,x\right\rangle )\overline{\alpha }(\sigma )-q_{1}(\sigma ))d\sigma
\end{equation*}

\begin{rem}
\label{rm:V0diff} Note that, with this choice of costs $C$, Assumptions \ref%
{asst1} are satisfied with $p>2$ so that, by Theorem \ref{th:mainnew}, the
value function $V_0$ is in $C^1_{Lip}$. By Lemma \ref{lm:fxptvintage} and
Theorem \ref{equgen} we then get that there exists a unique CLE-equilibrium
point, coinciding with the first component of the unique MP-equilibrium
point.\qn
\end{rem}

From this point on, one may {{procede}} as in the proof of
Theorem \ref{equgen} and Lemma \ref{theo:eq1} and compute CLE-equilibrium
points, once the data $\alpha ,q_{1},q_{0}$ are further specified.

\subsection{Power costs}

We now choose costs as in $(\ref{pc})$ and set $q=\frac{p}{p-1}$ . Note that
$p>2$ implies $q\in(1,2)$. The convex conjugate of the costs are then
\begin{equation}  \label{f^*}
f_{\beta }^{\ast }(v)=\left\{
\begin{array}{c}
(\beta p)^{1-q}q^{-1}v^{q}-\theta v+\beta \theta ^{p} \\
0%
\end{array}%
\begin{array}{c}
v\geq \beta p\theta ^{p-1} \\
v<\beta p\theta ^{p-1}%
\end{array}%
\right.
\end{equation}
with Lipschitz derivative
\begin{equation}  \label{f^*'}
\left( f_{\beta }^{\ast }\right) ^{\prime }(v)=\left\{
\begin{array}{c}
(\beta p)^{1-q}v^{q-1}-\theta \\
0%
\end{array}%
\begin{array}{c}
v\geq \beta p\theta ^{p-1} \\
v<\beta p\theta ^{p-1}.%
\end{array}%
\right.
\end{equation}

As a consequence the Legendre transform of $C$ is
\begin{equation*}
C^{\ast }(v)=f_{\beta _{0}}^{\ast }(v_{0}-q_{0})+\int_{0}^{\overline{s}%
}f_{\beta _{1}(s)}^{\ast }(v_{1}(s)-q_{1}(s))ds,
\end{equation*}%
which is a $C^{1}$ function with Lipschitz differential
\begin{equation*}
\left( C^{\ast }\right) ^{\prime }(v)(s)=\left( \left( f_{\beta _{0}}^{\ast
}\right) ^{\prime }(v_{0}-q_{0});\left( f_{\beta _{1}(s)}^{\ast }\right)
^{\prime }(v_{1}(s)-q_{1}(s))\right).
\end{equation*}%
Moreover
\begin{equation*}
T^{\theta }x(s)=\left( f_{\beta _{0}}^{\ast }\right) ^{\prime }(R^{\prime
}(\left\langle \alpha ,x\right\rangle )\overline{\alpha }(0)-q_{0})e^{-\mu
s}+\int_{0}^{s}e^{-\mu (s-\sigma )}\left( f_{\beta _{1}(\sigma )}^{\ast
}\right) ^{\prime }(R^{\prime }(\left\langle \alpha ,x\right\rangle )%
\overline{\alpha }(\sigma )-q_{1}(\sigma ))d\sigma .
\end{equation*}

\begin{rem}
  Note that Remark \ref{rm:V0diff} applies also to this case.\qn
\end{rem}

\bigskip

\section{Sensitivity analysis in two special cases}

\label{sens}

We here analyze further the case {{of}} linear-quadratic
costs discussed in Section \ref{sec:lqc}, and develop sensitivity analysis
accordingly. In particular we assume
\begin{equation}
\alpha (s)\equiv \alpha ,\ \ \ \beta _{1}(s)\equiv \beta _{0},\ \ \
q_{1}(s)=q_{0}e^{-ws}.  \label{acquisition_cost}
\end{equation}
Summing up, the objective functional of the profit maximizing firm is
\begin{eqnarray}
&&\int_{0}^{\infty }e^{-\lambda t}\left( R\left( Q\left( K(t)\right) \right)
-\int_{0}^{\bar{s}}\left( q_{1}\left( s\right) u_{1}\left( t,s\right) +\frac{%
1}{2}\beta _{0}u_{1}^{2}\left( t,s\right) \right) ds\right) dt
\label{objective} \\
&&\ \ \ \ \ \ \ \ \ \ \ \ \ \ \ \ \ \ \ \ \ \ \ \ \ \ \ \ \ \ \ \ \ \ \ \
-\int_{0}^{\infty }e^{-\lambda t}\left( q_{0}u_{0}\left( t\right) +\frac{1}{2%
}\beta _{0}u_{0}^{2}\left( t\right) \right) dt.  \notag
\end{eqnarray}%
We study separately the cases oflinear-quadratic and power revenues, depicted respectively in Lemma \ref{theo:eq1} $(i)$ and $(iii)$.
\subsection{Linear-Quadratic Revenues} \label{lqr} We here assume
\begin{equation}
R(Q)=bQ-aQ^{2}  \label{lin_quadr_R}
\end{equation}%
as in Lemma \ref{theo:eq1} $(i)$, so that the equilibrium distribution there described
equals
\begin{equation}
K^{\ast }\left( s\right) =-w_{2}\left( s\right) +\eta \;w_{1}\left( s\right)
,  \label{steadystate/general}
\end{equation}%
in which%
\begin{eqnarray}
w_{1}\left( s\right) &=&\frac{\alpha }{2\beta _{0}\left( \mu +\lambda
\right) }\left( e^{-\mu s}-e^{-\mu s}e^{-\left( \mu +\lambda \right) \bar{s}%
}+\frac{1-e^{-\mu s}}{\mu }+\right.  \label{w1} \\
&&\ \ \ \ \ \ \ \ \ \ \ \ \ \ \ \ \ \ \ \ \ \ \ \ \ \ \ \ \ \ \ \ \ \ \ \
\left. -\frac{e^{-\left( \mu +\lambda \right) \bar{s}}}{2\mu +\lambda }%
\left( e^{\left( \mu +\lambda \right) s}-e^{-\mu s}\right) \right) ,  \notag
\\
w_{2}\left( s\right) &=&\frac{q_{0}}{2\beta _{0}\left( \mu -w\right) }\left(
e^{-ws}-e^{-\mu s}\left( 1-\mu +w\right) \right) .  \label{w2}
\end{eqnarray}%
and
\begin{equation*}
\eta =\frac{b-2ac_{2}}{1+2ac_{1}}
\end{equation*}%
where
\begin{multline}  \label{k1}
c_{1} =\frac{\alpha ^{2}}{2\beta _{0}\left( \mu +\lambda \right) }\left(
\frac{1-\mu }{\mu ^{2}}\left( e^{-\mu \bar{s}}-1\right) +\frac{2\mu +\lambda
-1}{\mu \left( 2\mu +\lambda \right) }e^{-\left( 2\mu +\lambda \right) \bar{s%
}} \right. \\
\left. +\frac{\bar{s}}{\mu }-\frac{1}{\left( 2\mu +\lambda \right) \left(
\mu +\lambda \right) }+\frac{1-\left( \mu +\lambda \right) }{\mu \left( \mu
+\lambda \right) }e^{-\left( \mu +\lambda \right) \bar{s}}\right) ,
\end{multline}
\begin{equation}
c_{2} =\frac{\alpha q_{0}}{2\beta _{0}\left( \mu -w\right) }\left( \frac{\mu
-w-1}{\mu }\left( 1-e^{-\mu \bar{s}}\right) +\frac{1}{w}\left( 1-e^{-w\bar{s}%
}\right) \right) .  \label{k2}
\end{equation}

The explicit expression for the equilibrium distribution capital stock for every age, $%
K^*\left( s\right) ,$ allows us to obtain interesting economic implications.
To illustrate, we establish some numerical results, which mostly are
analytically proved as well. We start out from the following parameter
values:%
\begin{equation}
\alpha=3, \beta _{0}=0.5, \mu =0.2, \lambda=0.1, \bar s=10, q_{0}=5, w=0.25,
b=1, a=0.00004.  \label{benchmark}
\end{equation}%
The equilibrium distribution capital stock is depicted in Figure 1.

\begin{figure}
\begin{center}
\includegraphics[width=8truecm]{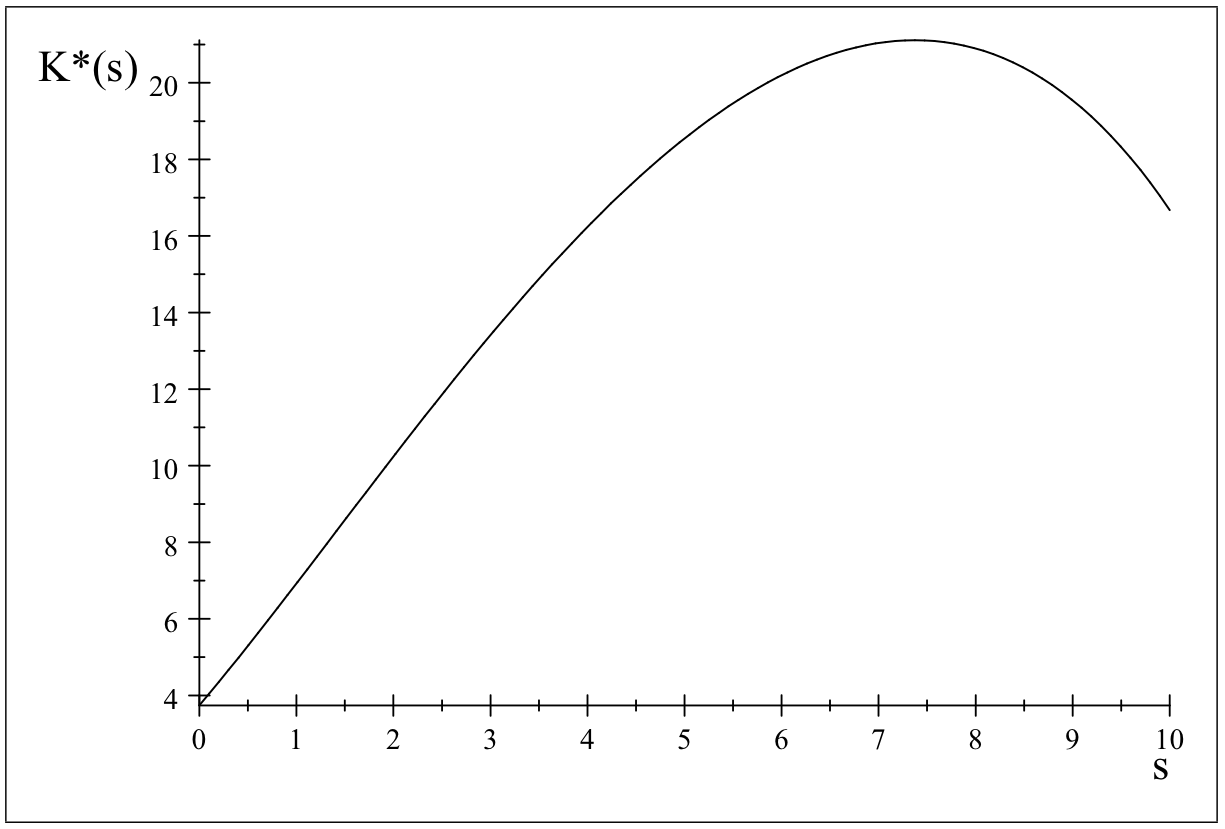}
\caption{Capital stock in equilibrium distribution for all ages, $s\in \left[ 0,\bar{s}\right] ,$ based on the parameter
values $\protect\alpha =3,\protect\beta _{0}=0.5,\protect\mu =0.2,\protect%
\lambda =0.1,\bar{s}=10,q_{0}=5,w=0.25,b=1,a=0.00004.$}
\end{center}
\end{figure}

We see that capital goods are non-monotonic with respect to age. To
understand this, Figure 2 depicts equilibrium distribution investment behavior, where
investment is given by
\begin{equation*}
u_{1}^{\ast }\left( s\right) =\frac{1}{2\beta _{0}}\left( \int_{s}^{\bar{s}%
}e^{-\left( \lambda +\mu \right) \left( j-s\right) }\left( b-2aQ^{\ast
}\right) \alpha dj-q_{0}e^{-ws}\right) ,
\end{equation*}%
with $Q^{\ast }$ being the production quantity in the equilibrium distribution.

\begin{figure}
\begin{center}
\includegraphics[width=8truecm]{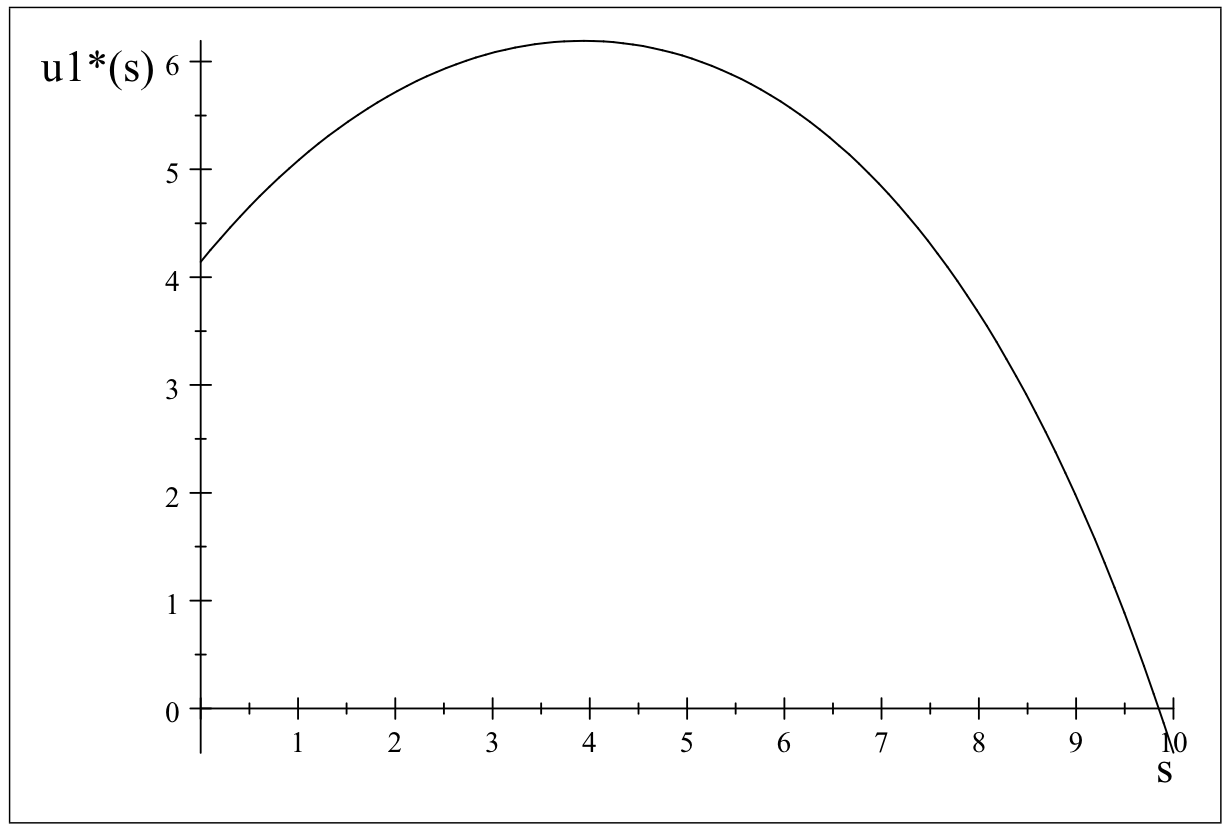}
\caption{Investment in equilibrium distribution for all ages, $s\in \left[ 0,\bar{s}\right] ,$ based on the parameter
values $\protect\alpha =3,\protect\beta _{0}=0.5,\protect\mu =0.2,\protect%
\lambda =0.1,\bar{s}=10,q_{0}=5,w=0.25,b=1,a=0.00004.$}
\end{center}
\end{figure}

Acquiring capital goods of older age is more attractive because they are
cheaper (see (\ref{acquisition_cost})). On the other hand their lifetime is
shorter, so they generate less revenue, which make older capital goods less
attractive. Figure 2 shows that the last effect dominates for the older
ages. The first effect plays a major role for younger ages. This makes sense
because the convexity of the unit cost of acquisition with respect to age,
as expressed in (\ref{acquisition_cost}), makes that these capital goods get
cheaper very quickly for slightly older age$.$

At first sight it is strange that $K^{\ast }\left( \bar{s}\right) >0$,
because $\bar{s}$ is the age capital goods are scrapped. However, the
presence of the adjustment costs,$\frac{1}{2}\beta _{0}\left[ u_{1}\left(
t,s\right) \right] ^{2}$, makes that it is not optimal to sell all capital
goods of age $\bar{s}$. In fact, convex adjustment costs make investments
continuous over time, and thus also over age since age and time go together.
Therefore, some of the capital goods of older age are still left. This is
confirmed in the investment graph of Figure 2, where we also see that $%
I^{\ast }\left( \bar{s}\right) $ is negative.

If we leave out the effect that older capital goods are less costly, 
the effect of having a shorter lifetime when capital goods get older remains, and steady state investments decrease with age.
This holds when we put%
\begin{equation}
q_{0}=q_{1}\left( s\right) =0,  \label{purely_quadratic}
\end{equation}%
and, combining this with (\ref{steadystate/general}), (\ref{w1}), and (\ref%
{w2}), we obtain, when revenue is not specified, that%
\begin{eqnarray}
K^{\ast }\left( s\right) &=&\eta w_{1}\left( s\right)  \label{K*(a)pqgeneral}
\\
&=&\frac{\eta \alpha }{2\beta _{0}\left( \mu +\lambda \right) }\left(
e^{-\mu s}-e^{-\mu s}e^{-\left( \mu +\lambda \right) \bar{s}}+\frac{1}{\mu }%
\left( 1-e^{-\mu s}\right) \right)  \notag \\
&&-\frac{\eta \alpha }{2\beta _{0}\left( \mu +\lambda \right) }\left(
e^{-\left( \mu +\lambda \right) \bar{s}}\frac{1}{2\mu +\lambda }\left(
e^{\left( \mu +\lambda \right) s}-e^{-\mu s}\right) \right) .  \notag
\end{eqnarray}%
In the specific case of a quadratic revenue function we get%
\begin{eqnarray}
&&  \label{K*(a)pq} \\
K^{\ast }\left( s\right) &=&\frac{bw_{1}\left( s\right) }{1+2ac_{1}}  \notag
\\
&=&\frac{b\left( \left( e^{-\mu s}-e^{-\mu s}e^{-\left( \mu +\lambda \right)
\bar{s}}+\frac{1}{\mu }\left( 1-e^{-\mu s}\right) -e^{-\left( \mu +\lambda
\right) \bar{s}}\frac{1}{2\mu +\lambda }\left( e^{\left( \mu +\lambda
\right) s}-e^{-\mu s}\right) \right) \right) }{\frac{2\beta _{0}\left( \mu
+\lambda \right) }{\alpha }+2a\alpha \left( \frac{1-\mu }{\mu ^{2}}\left(
e^{-\mu \bar{s}}-1\right) +\frac{2\mu +\lambda -1}{\mu \left( 2\mu +\lambda
\right) }e^{-\left( 2\mu +\lambda \right) \bar{s}}+\frac{1}{\mu }\bar{s}-%
\frac{1}{\left( 2\mu +\lambda \right) \left( \mu +\lambda \right) }+\frac{%
1-\mu -\lambda }{\mu \left( \mu +\lambda \right) }e^{-\left( \mu +\lambda
\right) \bar{s}}\right) }.  \notag
\end{eqnarray}
Equilibrium distribution investments being decreasing with age, 
also result in a hump-shaped structure of the steady state capital stock, like in Figure 1.
 The following proposition proves this analytically for a general
revenue function, thus based on the equilibrium distribution capital stock specified in (%
\ref{K*(a)pqgeneral}).

\begin{prop}
\label{sens:prop1}\textit{Consider the vintage capital stock model (\ref%
{ipde}), (\ref{Q}), ((\ref{objective})-(\ref{lin_quadr_R})) with purely
quadratic \ investment costs, i.e. we have (\ref{purely_quadratic}) that
partly replaces (\ref{acquisition_cost}). Then equilibrium distribution capital stock }$%
K^{\ast }\left( s\right) $ \textit{is positive for all }$s\in \left[ 0,\bar{s%
}\right] ,$ i\textit{s increasing in age for} $s\in \left[ 0,s^{\ast }\right]
$ \textit{and decreasing in age for }$s\in \left[ s^{\ast },\bar{s}\right] ,$
\textit{where}%
\begin{equation}
s^{\ast }=\frac{1}{2\mu +\lambda }\ln \left( \frac{2\mu +\lambda }{\mu
+\lambda }\left( \left( 1-\mu \right) e^{\left( \mu +\lambda \right) \bar{s}%
}+\mu \left( 1-\frac{1}{2\mu +\lambda }\right) \right) \right) >0.
\label{a*}
\end{equation}%
\textit{Furthermore, it holds that}%
\begin{equation}
K^{\ast }\left( 0\right) =\frac{\eta v}{2\beta _{0}\left( \mu +\lambda
\right) }\left( 1-e^{-\left( \mu +\lambda \right) \bar{s}}\right) >0,
\label{K*(0)pq}
\end{equation}%
\begin{equation}
K^{\ast }\left( \bar{s}\right) =\frac{\eta \alpha }{2\beta _{0}\left( \mu
+\lambda \right) }\left( -e^{-\mu \bar{s}}\left( -1+\frac{1}{\mu }\right)
+e^{-\mu \bar{s}}e^{-\left( \mu +\lambda \right) \bar{s}}\left( \frac{1}{%
2\mu +\lambda }-1\right) +\frac{1}{\mu }-\frac{1}{2\mu +\lambda }\right) >0.
\label{K*(h)pq}
\end{equation}
\end{prop}

\begin{proof}
From (\ref{K*(a)pq}) we obtain that%
\begin{equation*}
K^{\ast \prime }\left( s\right) =\frac{\eta \alpha }{2\beta _{0}\left( \mu
+\lambda \right) }e^{-\mu s}\left( \left( 1-\mu +\mu \left( 1-\frac{1}{2\mu
+\lambda }\right) e^{-\left( \mu +\lambda \right) \bar{s}}-\frac{\mu
+\lambda }{2\mu +\lambda }e^{-\left( \mu +\lambda \right) \left( \bar{s}%
-s\right) }e^{\mu s}\right) \right) ,
\end{equation*}%
from which it is straightforwardly concluded that $K^{\ast \prime }\left(
s\right) >0$ for $s<s^{\ast },$ with $s^{\ast }$ given by (\ref{a*}) and
vice versa.

To check whether $s^{\ast }$ is positive we need to show that%
\begin{equation*}
\frac{2\mu +\lambda }{\mu +\lambda }\left( \left( 1-\mu \right) e^{\left(
\mu +\lambda \right) \bar{s}}+\mu \left( 1-\frac{1}{2\mu +\lambda }\right)
\right) >1,
\end{equation*}%
which holds because%
\begin{eqnarray*}
&&\frac{2\mu +\lambda }{\mu +\lambda }\left( \left( 1-\mu \right) e^{\left(
\mu +\lambda \right) \bar{s}}+\mu \left( 1-\frac{1}{2\mu +\lambda }\right)
\right) \\
&>&\frac{2\mu +\lambda }{\mu +\lambda }\left( \left( 1-\mu \right) +\mu
\left( 1-\frac{1}{2\mu +\lambda }\right) \right) =1.
\end{eqnarray*}

From (\ref{K*(a)pqgeneral}) we straightforwardly obtain the expressions (\ref%
{K*(0)pq}) and (\ref{K*(h)pq}). To prove that $K^{\ast }\left( \bar{s}%
\right) >0$ we have to show that%
\begin{equation*}
-e^{-\mu \bar{s}}\left( -1+\frac{1}{\mu }\right) +e^{-\mu \bar{s}}e^{-\left(
\mu +\lambda \right) \bar{s}}\left( \frac{1}{2\mu +\lambda }-1\right) +\frac{%
1}{\mu }-\frac{1}{2\mu +\lambda }>0,
\end{equation*}%
which is true since%
\begin{eqnarray*}
&&-e^{-\mu \bar{s}}\left( -1+\frac{1}{\mu }\right) +e^{-\mu \bar{s}%
}e^{-\left( \mu +\lambda \right) \bar{s}}\left( \frac{1}{2\mu +\lambda }%
-1\right) +\frac{1}{\mu }-\frac{1}{2\mu +\lambda } \\
&>&1-\frac{1}{\mu }+\frac{1}{2\mu +\lambda }-1+\frac{1}{\mu }-\frac{1}{2\mu
+\lambda }=0.
\end{eqnarray*}
\end{proof}

As a final illustration of the interesting economic results that can be
obtained, let us focus on the impact of the productivity parameter $\alpha .$
Let us increase $\alpha $ from its original value $3$, as in (\ref{benchmark}%
), to $\alpha =12.$ The resulting equilibrium distribution capital stock is depicted in
Figure 3.
\begin{figure}
\begin{center}
\includegraphics[width=8truecm]{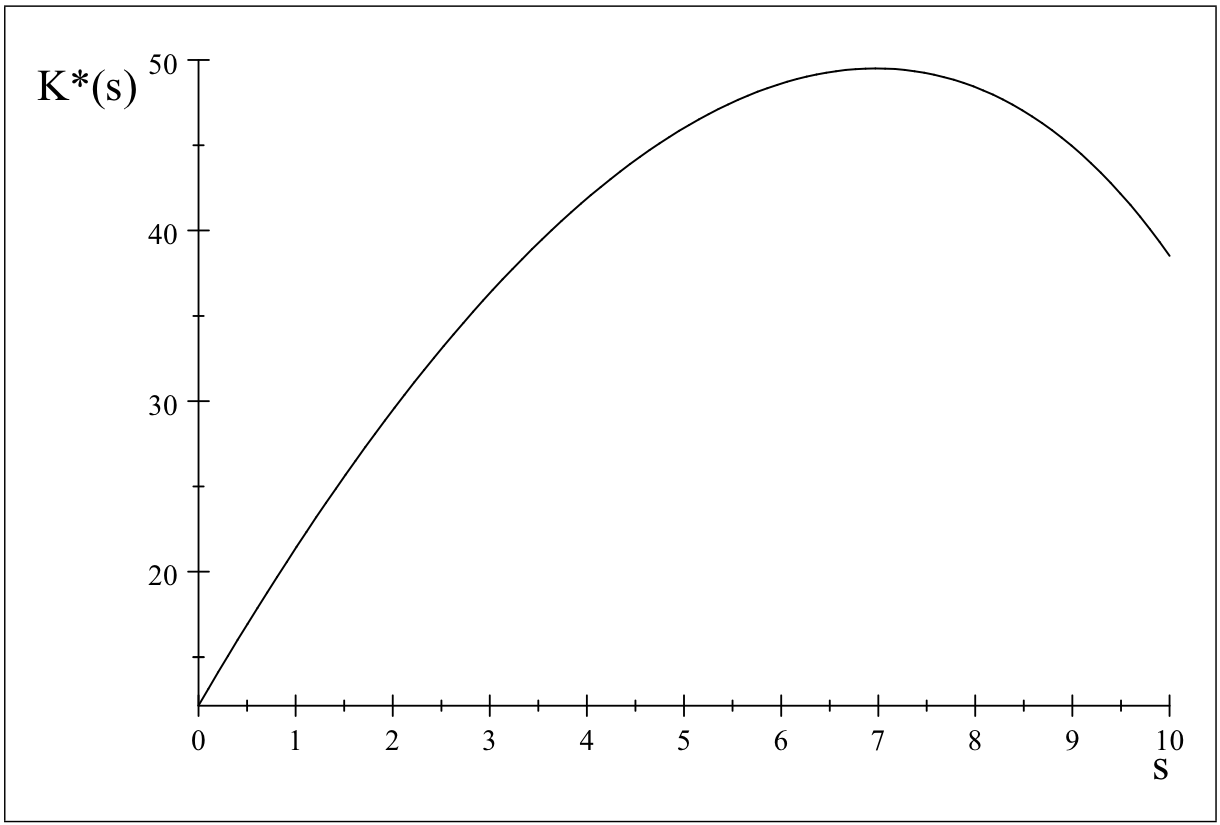}
\caption{Capital
stock in equilibrium distribution for every age $s\in \left[ 0,\bar{s}\right] $ for $\protect\alpha =12$
with remaining parameter values $\protect\beta _{0}=0.5,\protect\mu =0.2,%
\protect\lambda =0.1,\bar{s}=10,q_{0}=5,w=0.25,b=1,a=0.00004.$}
\end{center}
\end{figure}
The equilibrium distribution capital stock is still hump-shaped, as in the previous
figures, but the difference is that the firm buys more capital goods. Higher
productivity makes investing in capital goods more worthwhile. If we
increase the productivity parameter $\alpha $ even further to $\alpha =24,$
we obtain a equilibrium distribution capital stock being depicted in Figure 4.
\begin{figure}
\begin{center}
\includegraphics[width=8truecm]{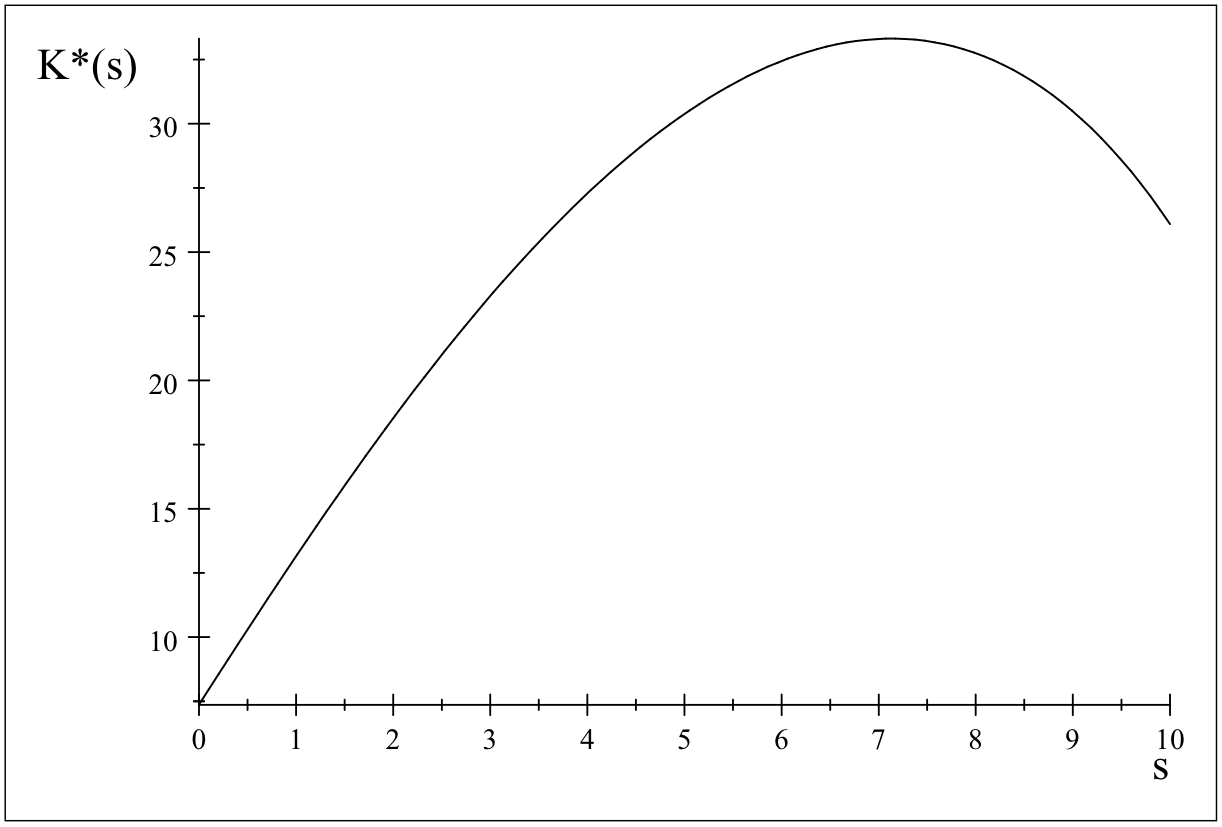}
\caption{Capital
stock in equilibrium distribution for all ages $s\in \left[ 0,\bar{s}\right] $ for productivity
parameter $\protect\alpha =24$ with remaining parameter values $\protect%
\beta _{0}=0.5,\protect\mu =0.2,\protect\lambda =0.1,\bar{s}%
=10,q_{0}=5,w=0.25,b=1,a=0.00004.$}
\end{center}
\end{figure}
Now the capital stock is smaller for all ages. Concavity of the revenue
function results in some bounded optimal quantity level, which, due to the
increased productivity, can be produced by less capital goods.

The non-monotonic behavior of the capital stock that is obtained when
productivity parameter $\alpha $ goes up, is an interesting result, from
which an expected outcome of including technological progress in the form of
process innovation can be predicted. Increased productivity first results in
more investments, but when productivity increases even further, investments
go down because the optimal quantity in this market can be produced by less
capital stock. The latter feature is new, and was for instance not derived
in Feichtinger et al. (2006).

The non-monotonicity dependence of the equilibrium distribution capital stock on the
productivity parameter can also be analytically proved in the special case
of purely quadratic investment costs, as we do in the next proposition.

\begin{prop}
\label{sens:prop2}\textit{In case of {{quadratic}} revenue}
\textit{(see} (\ref{lin_quadr_R})) \textit{and purely quadratic investment
costs, the equilibrium distribution} \textit{capital stock,} $K^{\ast }\left( s\right) $
\textit{is increasing with the productivity parameter} $\alpha $ \textit{for
}$\alpha \in \left[ 0,\hat{\alpha}\right] $ \textit{and decreasing with} $%
\alpha $ \textit{for }$\alpha \in \left[ \hat{\alpha},\infty \right) $%
\textit{, where}%
\begin{equation*}
\hat{\alpha}=\frac{1}{c_{1}b}\left( \sqrt{c_{2}+\frac{c_{1}b^{2}}{2a}}%
-c_{2}\right) >0,
\end{equation*}%
\textit{in which}%
\begin{eqnarray}
c_{1} &=&\frac{1}{2\beta _{0}\left( \mu +\lambda \right) }\left( \frac{1-\mu
}{\mu ^{2}}\left( e^{-\mu \bar{s}}-1\right) +\frac{2\mu +\lambda -1}{\mu
\left( 2\mu +\lambda \right) }e^{-\left( 2\mu +\lambda \right) \bar{s}%
}\right)  \label{c1} \\
&&+\frac{1}{2\beta _{0}\left( \mu +\lambda \right) }\left( \frac{\bar{s}}{%
\mu }-\frac{1}{\left( 2\mu +\lambda \right) \left( \mu +\lambda \right) }+%
\frac{1-\left( \mu +\lambda \right) }{\mu \left( \mu +\lambda \right) }%
e^{-\left( \mu +\lambda \right) \bar{s}}\right) ,  \notag
\end{eqnarray}%
\begin{equation}
c_{2}=\frac{q_{0}}{2\beta _{0}\left( \mu -w\right) }\left( \frac{\mu -w-1}{%
\mu }\left( 1-e^{-\mu \bar{s}}\right) +\frac{1}{w}\left( 1-e^{-w\bar{s}%
}\right) \right)  \label{c2}
\end{equation}
\end{prop}

\begin{proof}
From (\ref{steadystate/general})-(\ref{k2}), (\ref{c1}) and (\ref{c2}) we
obtain that
\begin{equation*}
\frac{\partial K^{\ast }}{\partial \alpha }=\frac{\partial }{\partial \alpha
}\left( \frac{b\alpha -2ac_{2}\alpha ^{2}}{1+2ac_{1}\alpha ^{2}}\right)
\varphi _{1}\left( s\right) ,
\end{equation*}%
in which%
\begin{eqnarray*}
\varphi _{1}\left( s\right) &=&\frac{1}{2\beta _{0}\left( \mu +\lambda
\right) }\left( e^{-\mu s}-e^{-\mu s}e^{-\left( \mu +\lambda \right) \bar{s}%
}+\frac{1}{\mu }\left( 1-e^{-\mu s}\right) -e^{-\left( \mu +\lambda \right)
\bar{s}}\frac{1}{2\mu +\lambda }\left( e^{\left( \mu +\lambda \right)
s}-e^{-\mu s}\right) \right) \\
&>&0.
\end{eqnarray*}%
It follows that%
\begin{equation*}
\frac{\partial K^{\ast }}{\partial \alpha }=\frac{-2ac_{1}b\alpha
^{2}-4ac_{2}\alpha +b}{\left( 1+2ac_{1}\alpha ^{2}\right) ^{2}}\varphi
_{1}\left( s\right) .
\end{equation*}%
Recognizing that the concave second order polynomial%
\begin{equation*}
-2ac_{1}b\alpha ^{2}-4ac_{2}\alpha +b
\end{equation*}%
has a negative root and a positive root being equal to $\hat{\alpha},$ gives
the result of the proposition.
\end{proof}

\begin{rem}
Concerning the stability of the equilibrium distribution $K^*$, we observe that Remark \ref{stabqc}
applies here and may imply, depending on the value of the parameters, that $K^*$ is  locally, or even globally, stable.\end{rem}

\subsection{Power Revenues}
In this section  we derive that the
same result as that in Section \ref{lqr}, i.e. equilibrium distribution capital stock is hump-shaped in $\alpha ,$
can be established for an alternative revenue function based on the
iso-elastic inverse demand function%
\begin{equation*}
p=bQ^{-\frac{1}{\epsilon }},
\end{equation*}%
in which $\varepsilon >1$ is the demand elasticity. Then the revenue
function is%
\begin{equation*}
R\left( Q\right) =bQ^{\gamma }
\end{equation*}%
with $\gamma =1-1/\epsilon .$ Since this revenue function has infinite
derivative for $Q=0,$ it is not a $C^{1}$ function. Therefore, instead we
employ the revenue function%
\begin{equation}
R\left( Q\right) =b\left( \left( \theta +Q\right) ^{\gamma }-\theta  \right) ,  \label{revenue-appr}
\end{equation}%
which approximates $R\left( Q\right) =bQ^{\gamma }$ for $\theta $ small.
We obtain from Lemma \ref{theo:eq1} $(iii)$ that for the revenue function as defined in \eqref{revenue-appr} and the investment costs being purely quadratic, which implies that $c_2=0,$ the $\eta$ from 
(\ref{steadystate/general}) is implicitly determined by
\begin{equation}
\eta \left( \theta +\eta c_{1}\right) ^{1-\gamma }-b\gamma =0.
\label{eta_implicit}
\end{equation}%
To establish the effect of the productivity parameter $\alpha $ on $K^{\ast
}\left( s\right) $ in the case of iso-elastic demand, we first determine how
$\eta $ depends on $\alpha .$ To do so, we first obtain from (\ref{k1}) that
$c_{1}$ is a quadratic function of $\alpha :$%
\begin{equation*}
c_{1}=f\alpha ^{2}
\end{equation*}%
with%
\begin{eqnarray*}
f &=&\frac{1}{2\beta _{0}\left( \mu +\lambda \right) }\left( \frac{1-\mu }{%
\mu ^{2}}\left( e^{-\mu \bar{s}}-1\right) +\frac{2\mu +\lambda -1}{\mu
\left( 2\mu +\lambda \right) }e^{-\left( 2\mu +\lambda \right) \bar{s}%
}\right) \\
&&+\frac{1}{2\beta _{0}\left( \mu +\lambda \right) }\left( \frac{1}{\mu }%
\bar{s}-\frac{1}{\left( 2\mu +\lambda \right) \left( \mu +\lambda \right) }+%
\frac{1-\left( \mu +\lambda \right) }{\mu \left( \mu +\lambda \right) }%
e^{-\left( \mu +\lambda \right) \bar{s}}\right) \\
&>&0.
\end{eqnarray*}%
This implies that we can rewrite (\ref{eta_implicit}) into%
\begin{equation}
\eta \left( \theta +\eta f\alpha ^{2}\right) ^{1-\gamma }-b\gamma =0.
\label{eta_k1_spec}
\end{equation}%
From the implicit function theorem we obtain that%
\begin{equation*}
\frac{\partial \eta }{\partial \alpha }=-\frac{2\left( 1-\gamma \right)
f\eta \alpha }{\eta \left( 1-\gamma \right) +f\alpha ^{2}\eta +\theta }<0,
\end{equation*}%
whereas we also conclude from (\ref{eta_k1_spec}) that%
\begin{equation}
\lim_{\alpha \rightarrow \infty }\eta \left( \alpha \right) =0.
\label{eta_v_limit}
\end{equation}

Now we are ready to establish how $K^{\ast }\left( s\right) $ depends on $%
\alpha .$ From (\ref{w1}) and (\ref{K*(a)pqgeneral}) we get%
\begin{equation*}
K^{\ast }\left( s\right) =\eta \alpha \varphi \left( s\right) ,
\end{equation*}%
with%
\begin{equation*}
\varphi \left( s\right) =\frac{1}{2\beta _{0}\left( \mu +\lambda \right) }%
\left( (1-e^{-\left( \mu +\lambda \right) \bar{s}})e^{-\mu s}+\frac{%
1-e^{-\mu s}} {\mu } -\frac{e^{-\left( \mu +\lambda \right) \bar{s}}}{2\mu
+\lambda }\left( e^{\left( \mu +\lambda \right) s}-e^{-\mu s}\right) \right)
>0.
\end{equation*}%
Hence, we obtain%
\begin{eqnarray*}
\frac{\partial K^{\ast }\left( s\right) }{\partial \alpha } &=&\varphi
\left( s\right) \left( \frac{\partial \eta }{\partial \alpha }\alpha +\eta
\right) \\
&=&\varphi \left( s\right) \eta \left( -\frac{2\left( 1-\gamma \right)
f\alpha ^{2}}{\eta \left( 1-\gamma \right) +f\alpha ^{2}\eta +\theta }%
+1\right) \\
&=&\varphi \left( s\right) \eta \left( -\frac{2\left( 1-\gamma \right) f}{%
\frac{\eta }{\alpha ^{2}}\left( 1-\gamma \right) +f\eta +\frac{\theta }{%
\alpha ^{2}}}+1\right) .
\end{eqnarray*}%
Since $\frac{\eta }{\alpha ^{2}}\left( 1-\gamma \right) +f\eta +\frac{\theta
}{\alpha ^{2}}$ is decreasing in $\alpha ,$ and, due to (\ref{eta_v_limit}),
it also holds that
\begin{eqnarray*}
-\frac{2\left( 1-\gamma \right) f}{\frac{\eta }{\alpha ^{2}}\left( 1-\gamma
\right) +f\eta +\frac{\theta }{\alpha ^{2}}}+1 &>&0\text{ for }\alpha =0, \\
-\frac{2\left( 1-\gamma \right) f}{\frac{\eta }{\alpha ^{2}}\left( 1-\gamma
\right) +f\eta +\frac{\theta }{\alpha ^{2}}}+1 &<&0\text{ for }\alpha
\rightarrow \infty ,
\end{eqnarray*}

we can conclude that we have proved the following proposition.

\begin{prop}
\label{sens:prop3}\textit{In case of iso-elastic demand and purely quadratic
investment costs, the equilibrium distribution capital stock,} $K^{\ast }\left( s\right)
$\textit{, is increasing with the productivity parameter }$\alpha $ \textit{%
for }$\alpha \in \left[ 0,\hat{\alpha}\right] $\textit{, and decreasing with}
$\alpha $ \textit{for }$\alpha \in \left[ \hat{\alpha},\infty \right) $,
\textit{where} $\hat{\alpha}$ \textit{is implicitly given by}%
\begin{equation*}
-\frac{2\left( 1-\gamma \right) f\hat{\alpha}^{2}}{\eta \left( \hat{\alpha}%
\right) \left( 1-\gamma \right) +f\eta \left( \hat{\alpha}\right) +\theta }%
+1=0,
\end{equation*}%
\textit{with}%
\begin{eqnarray*}
f &=&\frac{1}{2\beta _{0}\left( \mu +\lambda \right) }\left( \frac{1-\mu }{%
\mu ^{2}}\left( e^{-\mu \bar{s}}-1\right) +\frac{2\mu +\lambda -1}{\mu
\left( 2\mu +\lambda \right) }e^{-\left( 2\mu +\lambda \right) \bar{s}}+%
\frac{1}{\mu }\bar{s}\right) \\
&&-\frac{1}{2\beta _{0}\left( \mu +\lambda \right) }\left( \frac{1}{\left(
2\mu +\lambda \right) \left( \mu +\lambda \right) }+\frac{1-\left( \mu
+\lambda \right) }{\mu \left( \mu +\lambda \right) }e^{-\left( \mu +\lambda
\right) \bar{s}}\right) \\
&>&0.
\end{eqnarray*}
\end{prop}

\begin{rem}
Concerning the stability of the equilibrium distribution $K^*$, we observe that Lemma \ref{stablin} 
applies here and may imply, depending on the value of the parameters, that $K^*$ is  locally, or even globally, stable.\end{rem}
\appendix

\section{Proofs of Subsection \protect\ref{subsec:MP}}

We here present a detailed description of the material in Subsection \ref%
{subsec:MP} as well as all the proofs of the results there stated. Firstly
we note that solutions of the ODEs in (\ref{CO}) have to be intended in mild
form. That is expressed in the following definitions.

\begin{defi}
\label{df:mild} Let Assumptions \ref{asst2} be satisfied. Let $t\ge 0$ and
let $u\in L_\lambda^p(t,+\infty;U)$. The mild solution of (\ref{eq:statoV'})
is the function $y\in L^1_{loc}(t,+\infty;V^\prime)$ given by
\begin{equation}  \label{mildES}
y(\tau)=e^{(\tau-t)A}x+\int_t^\tau e^{(\tau-\sigma)A}Bu(\sigma) d\sigma,\ \
\tau\in[t,+\infty[ .\
\end{equation}
The mild solution of (\ref{cost})-(\ref{tc}) is the function $%
\pi:[t,+\infty[ \to V$ given by
\begin{equation}  \label{pmild}
\pi(\tau)=\int_\tau^{+\infty}
e^{(A_1^*-\lambda)(\sigma-\tau)}g_0^\prime(y(\sigma))d\sigma.
\end{equation}
A mild solution to the closed loop equation (\ref{CLE}) is a function $y\in
L^1_{loc}(t,+\infty;V^\prime)$ satisfying
\begin{equation}  \label{CLEmild}
y(\tau)=e^{(\tau-t)A}x+\int_t^\tau e^{(\tau-\sigma)A} B(h_0^*)^{\prime} {%
(-B^*\Psi^\prime(y(s))) }d\sigma,\ \ \tau\in[t,+\infty[ .\
\end{equation}%
\end{defi}

\bigbreak
\color{black}

\begin{lemma}
\label{p^*def} Let Assumption \ref{asst2} hold, assume $p\geq 2,$ $q=\frac{p%
}{p-1}$ and $\lambda >(2\omega) \vee \omega $. Let also $u \in
L_\lambda^p(t,+\infty;U)$. Then:

$(i)$ $\pi $ given by (\ref{pmild}) is well defined and belongs to $%
C^{0}([t,+\infty[;V)$;

$(ii)$ if $p> 2$ and $\lambda>0$, then $\pi\in L^q_\lambda(t,+\infty;V)$;

$(iii)$ if $p=2$ and $\lambda>0$, then $\pi\in
L^2_{\lambda+\varepsilon}(t,+\infty;V)\cap L^2(t,T;V),\ \forall \ T<+\infty,
\ \varepsilon>0$;

$(iv)$ if $p=2$, $\lambda>0$ and $\omega <0$, then $\pi\in
L^2_{\lambda}(t,+\infty;V).$
\end{lemma}

\begin{proof}

We first prove $(i)$. Note that by assumptions on $g_0^\prime$, the
integrand in (\ref{pmild}) can be estimated as follows
\begin{equation*}
\vert e^{(A_1^*-\lambda)(\sigma-\tau)} g_0^\prime(y(\sigma))\vert_V\le \vert
g_0^\prime\vert_{\mathcal{B}_1}e^{-(\lambda-\omega)(\sigma-\tau)}(1+\vert
y(\sigma)\vert_{_{V^\prime}}).
\end{equation*}
Since $\lambda >\omega$, to prove the first assertion,  i.e. that $\pi(\cdot)
$ is well defined, it is enough to show that, for every $\tau\ge t$, the map
$\sigma\mapsto e^{-(\lambda-\omega)(\sigma-\tau)}\vert
y(\sigma)\vert_{_{V^\prime}}$ is in $L^1(\tau,+\infty)$.

By Assumption \ref{asst2} (see also \cite[Lemma 4.5]{FaGo2}) one has
\begin{equation}  \label{eq:stimanormay}
\begin{split}
\vert y(\sigma)\vert_{_{V^\prime}} & \le e^{\omega (\sigma-t)}\vert
x\vert_{_{V^\prime}} +\int_t^\sigma e^{\omega (\sigma- r)}\|B\|_{\mathcal{L}%
(U,V^{\prime })}\vert u(r)\vert_Udr \\
&\le Ce^{\omega \sigma}\left(\vert x\vert_{_{V^\prime}} +\int_t^\sigma
e^{-\omega r}\vert u(r)\vert_Udr\right) \\
&\le C_1e^{\omega \sigma}(1+\rho(t,\sigma)^{\frac{1}{q}})
\end{split}%
\end{equation}
for suitable constants $C$ (depending only on $t$) and $C_1$ (depending on $t
$, $x$, and $u$), where
\begin{equation*}
\rho(t,\sigma)=%
\begin{cases}
\vert e^{q\left(\frac{\lambda}{p}-\omega\right)t}- e^{q\left(\frac{\lambda}{p%
}-\omega\right)\sigma}\vert & \lambda\not=\omega p \\
\vert t-\sigma\vert & \lambda=\omega p.%
\end{cases}%
\end{equation*}
Hence
\begin{equation*}
e^{-(\lambda-\omega)\sigma}\vert y(\sigma)\vert_{_{V^\prime}}\le
C_1e^{-(\lambda-2\omega)\sigma}(1+\rho(t,\sigma)^{\frac{1}{q}}),
\end{equation*}
so that in the case $\lambda\not=\omega p$ one obtains
\begin{equation}  \label{rho31}
e^{-(\lambda-2\omega)\sigma}\rho(t,\sigma)^{\frac{1}{q}} = \vert
e^{-q(\lambda-2\omega)\sigma}e^{q\left(\frac{\lambda}{p}-\omega%
\right)t}-e^{-(\lambda-q\omega)\sigma}\vert^{\frac{1}{q}} \le C_2\left[%
e^{-(\lambda-2\omega)\sigma} \vee e^{-\frac{1}{q}(\lambda-q\omega)\sigma}%
\right]
\end{equation}
for a suitable constant $C_2$, whereas in the case $\lambda=\omega p$ one
has
\begin{equation}  \label{rho32}
e^{-(\lambda-2\omega)\sigma}\rho(t,\sigma)^{\frac{1}{q}}=
e^{-(\lambda-2\omega)\sigma}\vert t-\sigma\vert^{\frac{1}{q}}.
\end{equation}
Since $\lambda>(2\omega) \vee \omega $ then also $\lambda > q \omega$ since $%
q=p/(p-1)\in (1,2]$. Hence, for each $\tau \ge t$, the integrand in %
\eqref{pmild} is in $L^1(\tau,+\infty;V)$.

The proof that $\pi\in C^0([t,+\infty);V)$ follows from the dominated
convergence theorem and the fact that the above estimates does not depend on
$\tau$, when $\tau$ is taken in any bounded interval.

\medskip

Now we prove $(ii)$. We start by showing that  $p>2$ implies $\pi\in
L^q_\lambda(t,+\infty;V)$. From the estimates above, one has
\begin{equation}
\begin{split}  \label{rhonew}
\vert \pi(\tau)\vert_V&\le \frac{\vert g_0^\prime\vert_{\mathcal{B}_1}}{%
\lambda-\omega}+ \frac{\vert g_0^\prime\vert_{\mathcal{B}_1} C_1
e^{\omega\tau}}{\lambda-2\omega}+\vert g_0^\prime\vert_{\mathcal{B}_1} C_1
e^{(\lambda-\omega)\tau}
\int_\tau^{+\infty}e^{-(\lambda-2\omega)\sigma}\rho(t,\sigma)^{\frac{1}{q}%
}d\sigma \\
&\equiv \gamma_1(\tau)+\gamma_2(\tau)+\gamma_3(\tau).
\end{split}%
\end{equation}
The function $\gamma_1$ is trivially in $L^q_\lambda(t,+\infty;\mathbb{R})$,
since $\lambda>0$. The function $\gamma_2$ is in $L^q_\lambda(t,+\infty;%
\mathbb{R})$ since, as observed above, it must be $\lambda > q \omega$ as $%
q\in (1,2]$. Regarding $\gamma_3$, in the case $\lambda=\omega p$, it must
be necessarily $\omega>0$ (if not we cannot have $\lambda>\omega$). Let then
$\delta>0$ such that $\lambda>2\omega+2\delta$. Since, by simple
computations, $(\sigma -t)^{1/q}e^{-\delta \sigma} \le
\sigma^{1/q}e^{-\delta \sigma}\le (qe\delta)^{-1/q}$, we then have
\begin{equation*}
\begin{split}
\int_\tau^{+\infty}e^{-(\lambda-2\omega)\sigma}\vert t-\sigma\vert^{\frac{1}{%
q}}d\sigma \le
(qe\delta)^{-1/q}\int_\tau^{+\infty}e^{-(\lambda-2\omega-\delta)\sigma}d%
\sigma \le (qe\delta)^{-1/q}\frac{e^{-(\lambda-2\omega-\delta)\tau}}{%
\lambda-2\omega-\delta}
\end{split}%
\end{equation*}
Hence, in this case, for a suitable constant $C_3$ one has
\begin{equation}  \label{rho33}
e^{-\lambda\tau}\gamma_3(t)^q \le C_3
e^{-\lambda\tau}e^{q(\lambda-\omega)\tau} e^{-q(\lambda-2\omega-\delta)\tau}
= C_3 e^{-[\lambda-(\omega+\delta) q]\tau};
\end{equation}
the last is an integrable function in $[t,+\infty)$ by the choice of $\delta$%
. In the case $\lambda\not=\omega p$, by means of (\ref{rho31}) one has
\begin{equation}  \label{rho33new}
\int_\tau^{+\infty}e^{-(\lambda-2\omega)\sigma}\rho(t,\sigma)^{\frac{1}{q}%
}d\sigma \le C_2 \int_\tau^{+\infty} \left[e^{-(\lambda-2\omega)\sigma}\vee
e^{-\frac1q(\lambda-q\omega)\sigma}\right]d\sigma.
\end{equation}
By \eqref{rhonew} we then have, for a suitable constant $C_4>0$,
\begin{equation}
\begin{split}  \label{rho34}
e^{-\lambda\tau}\gamma_3(\tau)^q \le C_4
e^{-\lambda\tau}e^{q(\lambda-\omega)\tau} \left[e^{-(\lambda-2\omega)\tau}%
\vee e^{-\frac1q(\lambda-q\omega)\tau}\right]^q \le C_4 \left[%
e^{-(\lambda-q\omega)\tau}\vee e^{-\lambda(2-q)\tau}\right],
\end{split}%
\end{equation}
which implies $\gamma_3\in L^q_\lambda(t,+\infty;\mathbb{R})$ also in this
case.

\medskip To prove $(iii)$ it sufficies to observe that $p=2$ then (\ref%
{rho33})-(\ref{rho34}) computed with $q=2$ and $\lambda+\varepsilon$ in
place of $\lambda$ imply promptly $\pi\in
L^2_{\lambda+\varepsilon}(t,+\infty;V)$, $\forall\varepsilon>0$.

\medskip

Finally we prove $(iv)$. Let $p=2$. Observe that, for $\tau\ge t$,
\begin{equation*}
|\pi(\tau)|_V = \left| \int_\tau^{+\infty} e^{(A_1^*-\lambda)(\sigma-\tau)}
g_0^\prime(y(\sigma))d\sigma \right|_V \le \vert g_0^\prime\vert_{\mathcal{B}%
_1} \int_\tau^{+\infty} e^{-(\lambda-\omega)(\sigma-\tau)} (1+\vert
y(\sigma)\vert_{_{V^\prime}})d\sigma.
\end{equation*}
which implies
\begin{equation*}
e^{-\lambda \tau}|\pi(\tau)|_V^2 \le e^{-\lambda \tau} 2\vert
g_0^\prime\vert_{\mathcal{B}_1}^2 \left[\frac{1}{(\lambda-\omega)^2}
+\left(\int_\tau^{+\infty} e^{-(\lambda-\omega)(\sigma-\tau)} \vert
y(\sigma)\vert_{_{V^\prime}}d\sigma\right)^2\right].
\end{equation*}
Since $\lambda>0$, the first term is integrable on $[t,+\infty[$. Concerning
the second term we exploit Jensen's inequality to get
\begin{equation}  \label{eq:Jensen1}
\left(\int_\tau^{+\infty} e^{-(\lambda-\omega)(\sigma-\tau)} \vert
y(\sigma)\vert_{_{V^\prime}}d\sigma\right)^2 \le \frac{1}{\lambda-\omega}%
\int_\tau^{+\infty} e^{-(\lambda-\omega)(\sigma-\tau)} \vert
y(\sigma)\vert_{_{V^\prime}}^2d\sigma
\end{equation}
Hence
\begin{equation}  \label{eq:intpi}
\int_t^{+\infty}e^{-\lambda \tau}|\pi(\tau)|_V^2\,d\tau \le
\int_t^{+\infty}e^{-\lambda \tau} 2\vert g_0^\prime\vert_{\mathcal{B}_1}^2 %
\left[\frac{1}{(\lambda-\omega)^2} + \frac{1}{\lambda-\omega}%
\int_\tau^{+\infty} e^{-(\lambda-\omega)(\sigma-\tau)} \vert
y(\sigma)\vert_{_{V^\prime}}^2 d\sigma \right]d\tau.
\end{equation}
The first term of the right hand side is $2\vert g_0^\prime\vert_{\mathcal{B}%
_1}^2 e^{-\lambda t}/(\lambda-\omega)^2$. To estimate the second we use
Fubini-Tonelli Theorem (see e.g. \cite[Theorem 1.33]{FabbriGozziSwiech17}),
recalling that the integrand here is positive. Indeed
\begin{equation*}
\int_t^{+\infty}e^{-\lambda \tau} \int_\tau^{+\infty}
e^{-(\lambda-\omega)(\sigma-\tau)} \vert y(\sigma)\vert_{_{V^\prime}}^2
d\sigma d\tau = \int_t^{+\infty} \vert y(\sigma)\vert_{_{V^\prime}}^2
e^{-(\lambda-\omega)\sigma} \int_t^{\sigma} e^{-\lambda \tau}
e^{(\lambda-\omega)\tau} d\tau d\sigma
\end{equation*}

\begin{equation*}
=\int_t^{+\infty} \vert y(\sigma)\vert_{_{V^\prime}}^2
e^{-(\lambda-\omega)\sigma} \frac{e^{-\omega t}-e^{-\omega \sigma}}{\omega}
\,d\sigma=:I_0 
\end{equation*}
where, in the last step, we use that $\omega \ne 0$. To prove the claim it
is now enough to prove that the last integral $I_0$ is finite.
First of all, by \eqref{eq:stimanormay}, we have
\begin{equation}  \label{eq:stimanormay2}
\vert y(\sigma)\vert_{_{V^\prime}}^2 \le C_5e^{2\omega \sigma} \left[1
+\left(\int_t^\sigma e^{-\omega r}\vert u(r)\vert_U \,dr\right)^2\right] \\
\end{equation}
for a suitable constant $C_5$ (depending on $t$, $x$, and $u$). We apply
again Jensen's inequality getting
\begin{equation*}
\left(\int_t^\sigma e^{-\omega r}\vert u(r)\vert_U \,dr\right)^2 \le \frac{%
e^{-\omega t}-e^{-\omega \sigma}}{\omega} \int_t^\sigma e^{-\omega r}\vert
u(r)\vert_U^2 \,dr
\end{equation*}
Then we have
\begin{equation}  \label{eq:stimaexpy}
\begin{split}
I_0&\le C_5\left[ \int_t^{+\infty} e^{-(\lambda-3\omega)\sigma} \frac{%
e^{-\omega t}-e^{-\omega \sigma}}{\omega}\,d\sigma \right. \\[2mm]
&\left. +\int_t^{+\infty} e^{-(\lambda-3\omega)\sigma} \left(\frac{%
e^{-\omega t}-e^{-\omega \sigma}}{\omega}\right)^2 \int_t^\sigma e^{-\omega
r}\vert u(r)\vert_U^2 \,dr\, d\sigma\right]
\end{split}%
\end{equation}
Since $\lambda >0$ and $\omega<0$ the first integral of the right hand side
of \eqref{eq:stimaexpy} is finite, positive, and, its value is
\begin{equation*}
e^{-(\lambda-2\omega)t}\frac1\omega \left[\frac{1}{\lambda-3\omega}-\frac{1}{%
\lambda-2\omega}\right]>0.
\end{equation*}
Concerning the second integral of the right hand side of \eqref{eq:stimaexpy}
(recalling that the integrand is positive) we apply Fubini-Tonelli Theorem
again, to get that it is equal to
\begin{equation*}
I_1:=\int_t^{+\infty}e^{-\omega r}\vert u(r)\vert_U^2 \int_{r}^{+\infty}
e^{-(\lambda-3\omega)\sigma} \left(\frac{e^{-\omega t}-e^{-\omega \sigma}}{%
\omega}\right)^2 d\sigma dr
\end{equation*}
At this point we really need to use that $\omega<0$\footnote{%
This was not needed up to now. Above we only used the fact that $\lambda
>3\omega$ and, to simplify computations, $\omega \ne 0$.}. This implies that
the squared fraction above is smaller than $e^{-2\omega \sigma}/(\omega^2)$.
Hence we have
\begin{equation*}
I_1 \le \frac{1}{\omega^2} \int_t^{+\infty}e^{-\omega r}\vert
u(r)\vert_U^2\, \int_{r}^{+\infty} e^{-(\lambda-\omega)\sigma}d\sigma dr \le
\frac{1}{\omega^2(\lambda-\omega)} \int_t^{+\infty}e^{-\lambda r}\vert
u(r)\vert_U^2\, dr 
\end{equation*}
Since $u\in L^2_\lambda(t,+\infty;U)$ the above imply the finiteness of $I_1$
and, consequently, of $I_0$, which proves the claim. %
\end{proof}

\color{black}

%
%
%
%

\begin{theo}
{Let Assumption \ref{asst2} hold, let $p\geq 2,$ $q=\frac{p}{p-1}$ and $%
\lambda >(2\omega) \vee \omega $. Let also $u \in L_\lambda^p(t,+\infty;U)$.}
If $\pi\in W^{1,1}(t,+\infty; V)$ satisfies $(\ref{cost})$ almost everywhere
in $[t,+\infty)$ and the transversality condition $(\ref{tc})$, then $\pi$
is given by (\ref{pmild}), that is $\pi$ is the mild solution of (\ref{cost}%
)-(\ref{tc}).
\end{theo}

\begin{proof}
By variation of constants formula, any $\pi$ satisfying $(\ref{cost})$ a.e.
must also satisfy
\begin{equation}  \label{addi}
\pi(\tau)=e^{(A_1^*-\lambda) (T-\tau)}\pi(T)+\int_\tau^{T}
e^{(A_1^*-\lambda)(\sigma-\tau)}g_0^\prime(y(\sigma))d\sigma, \ \forall
T>t,\ \forall \tau\in[t,T].
\end{equation}
Note that $(\ref{tc})$ implies
\begin{equation*}
\lim_{T\to+\infty}\vert e^{(A_1^*-\lambda)
(T-\tau)}\pi(T)\vert_V\le\lim_{T\to+\infty} e^{(\omega-\lambda)
(T-\tau)}\vert \pi(T)\vert_V=0.
\end{equation*}
hence by passing to limits as $T\to+\infty$ in (\ref{addi}) one derives
\begin{equation*}
\pi(\tau)=\lim_{T\to+\infty}\int_\tau^{T}
e^{(A_1^*-\lambda)(\sigma-\tau)}g_0^\prime(y(\sigma))d\sigma=\int_\tau^{+%
\infty} e^{(A_1^*-\lambda)(\sigma-\tau)}g_0^\prime(y(\sigma))d\sigma
\end{equation*}
where the last equality follows from estimates (\ref{rho31})-(\ref{rho32}).
\end{proof}

We are now ready to prove the Maximum Principle.

\bigskip

\noindent\textbf{Proof of Theorem \ref{Pmp} (Maximum Principle)}

Let $K:L^p_\lambda(t,+\infty;U)\to \mathbb{R}\cup\{+\infty\}$, and $%
G:L^p_\lambda(t,+\infty;U)\to \mathbb{R}\cup\{+\infty\}$ be defined by
\begin{equation*}
K(u):= \int_t^{+\infty} e^{-\lambda\tau}h_0(u(\tau))d\tau, \qquad
G(u):=\int_t^{+\infty}e^{-\lambda\tau}g_0(y(\tau;t,x,u))d\tau
\end{equation*}
so that for all $u\in dom(K)\cap dom(G)$ we have
\begin{equation}  \label{eq:subdiffinclusion}
J(u)=K(u)+G(u),\ and\ \partial J(u)\supseteq \partial K(u)+\partial G(u).
\end{equation}

\emph{Claim 1:} {For} $u \in int\, dom(K)$ we have
\begin{equation}  \label{subdiff2}
\begin{split}
\partial K(u)&=S,\ \ where \\
S\equiv\{ \varphi\in L^q_\lambda(t,+\infty;U)\ &:\ \varphi(\tau)\in \partial
h_0(u(\tau)),\ for \; a.e.\ \tau\in[t,+\infty)\} \\
\end{split}%
\end{equation}
Indeed
\begin{equation}  \label{subdiff}
\begin{split}
\partial K(u)&=\bigg\{\varphi\in L^q_\lambda(t,+\infty;U)\ :\  \\
&\int_t^{+\infty} {e^{-\lambda \tau}}\big[h_0(w(\tau))-h_0(u(\tau))
-(\varphi(\tau)\vert w(\tau)-u(\tau))_U\big]d\tau\ge0,\ \forall w\in
L^p_\lambda(t,+\infty;U)\bigg\},
\end{split}%
\end{equation}
so that $S\subset\partial K(u)$ is straightforward. To show the reverse
inclusion, we let $\varphi$ be any fixed element of $\partial K(u)$, $E$ any
measurable subset of $[t,+\infty)$, and we set, for any $w\in
L^p_\lambda(t,+\infty;U)$,
\begin{equation}
\tilde w(\tau)=
\begin{cases}
u(\tau), & \tau\not\in E \\
w(\tau), & \tau\in E%
\end{cases}%
\end{equation}
Clearly we still have $\tilde w \in L^p_\lambda(t,+\infty;U)$, hence we
derive
\begin{equation*}
\int_{E} {e^{-\lambda \tau}}\big[h_0(w(\tau))-h_0(u(\tau))
-(\varphi(\tau)\vert w(\tau)-u(\tau))_U\big]d\tau\ge0,\ \forall w\in
L^p_\lambda(t,+\infty;U).
\end{equation*}
Since $E$ and $w$ where arbitrarily chosen, the above implies
\begin{equation*}
h_0(w(\tau))-h_0(u(\tau)) -(\varphi(\tau)\vert w(\tau)-u(\tau))_U\ge0,\ \
for \; a.e. \ \tau\in[0,+\infty)
\end{equation*}
that is, $\varphi(\tau)\in\partial h_0(\tau)$ for almost every $\tau \ge0$.

\bigskip

\emph{Claim 2: Let $(u,y)$ be admissible at $(t,x)$. Assume that there
exists $\pi \in L^q_\lambda(t,+\infty;V)$ such that $(\pi,u,y)$ satisfies (%
\ref{CO}). Then $(u,y)$ is optimal at $(t,x)$.}

\medskip

{Let
$v $ be any control in $dom(G)$.} Then
\begin{equation}  \label{eq:Gv-Gu}
\begin{split}
{G(v)-G(u)}&=\int_t^{+\infty}\left[g_0(y(\tau;v))-g_0(y(\tau;u))\right]
e^{-\lambda \tau}d\tau \\
&\ge \int_t^{+\infty}\langle g_0^\prime(y(\tau;u)),\int_t^\tau
e^{(\tau-\sigma)A}B (v(\sigma)-u(\sigma))d\sigma\rangle e^{-\lambda
\tau}d\tau \\
&=\int_t^{+\infty}\int_\sigma^{+\infty} \big( B^*e^{(\tau-\sigma)A_1^*}g_0^%
\prime(y(\tau;u))\big\vert v(\sigma)-u(\sigma)\big)_U e^{-\lambda\tau}d\tau
d\sigma \\
&=\int_t^{+\infty} \big(\int_\sigma^{+\infty}
B^*e^{(\tau-\sigma)(A_1^*-\lambda)}g_0^\prime(y(\tau;u))d \tau \big\vert %
v(\sigma)-u(\sigma)\big)_U e^{-\lambda\sigma}d\sigma \\
&=\int_t^{+\infty}\big( B^* \pi (\sigma)\big\vert v(\sigma)-u(\sigma)\big)_U
e^{-\lambda\sigma}d\sigma \\
&=\langle B^*\pi, v-u\rangle_{L^q_\lambda (t,+\infty;U),
L^p_\lambda(t,+\infty;U)}
\end{split}%
\end{equation}
where we could exchange the order of integration since $\pi$ is in $%
L^q_\lambda(t,+\infty;V)$ by assumption. Then we proved that
\begin{equation*}
B^*\pi\in \partial G(u).
\end{equation*}
{Now, by \eqref{CO}} 
we also know that $-B^*\pi(\sigma)\in\partial h_0(u(\sigma))$ almost
everywhere in $[t,+\infty[$; hence, by Claim 1, we get $-B^*\pi\in \partial
K(u)$. By \eqref{eq:subdiffinclusion} it follows that $\partial J(u)\ni 0$,
that is, $u$ is optimal and $(i)$ is proved.

\smallskip

\emph{Claim 3. Assume that, either $p>2$ and $\lambda>0$, or $p=2$, $%
\lambda>0$ and $\omega <0$. Assume that $(u^*,y^*)$ is optimal at $(t,x)$,
and let $\pi^*$ be the associated solution of \eqref{pmild}. Then $G$ is
G\^ateaux differentiable in $u^*$ with $G^\prime(u^*)=B^*\pi^*$.
Consequently $\partial J(u^*)=B^*\pi^* + S$.}

\medskip

First of all we recall that, by assumption and by Lemma \ref{p^*def}, we
have $\pi^*\in L^q_\lambda(t,+\infty;V)$. Moreover, for any fixed $v$ in $%
L^p_\lambda(t,+\infty;U)$, and any $\epsilon>0$, there exists $%
0<\epsilon_0\le\epsilon$ such that
\begin{equation*}
\begin{split}
&\frac{G(u^*+\epsilon v)-G(u^*)}{\epsilon} = \int_t^{+\infty} \frac{%
g_0(y(\tau;u^*+\epsilon v))-g_0(y^*(\tau))}{\epsilon} e^{-\lambda \tau}d\tau
= \\
&=\int_t^{+\infty}\left\langle g_0^\prime(y(\tau;u^*+\epsilon_0 v)),
y(\tau;u^*+\epsilon_0 v)- y^*(\tau)\right\rangle e^{-\lambda \tau}d\tau \\
&= \int_t^{+\infty}\left\langle g_0^\prime(y(\tau;u^*+\epsilon_0 v)),
\int_t^\tau e^{(\tau-\sigma)A}B v(\sigma)d\sigma \right\rangle e^{-\lambda
\tau}d\tau
\end{split}%
\end{equation*}
Hence, arguing as in \eqref{eq:Gv-Gu} to rewrite the term $\langle
B^*\pi^*,v\rangle_{L^q_\lambda, L^p_\lambda}$, we get \color{black}
\begin{equation*}
\begin{split}
&\left\vert\frac{G(u^*+\epsilon v)-G(u^*)}{\epsilon} -\langle
B^*\pi^*,v\rangle_{L^q_\lambda, L^p_\lambda}\right\vert= \\
&= \left\vert\int_t^{+\infty}\frac{g_0(y(\tau;u^*+\epsilon v))-g_0(y^*(\tau))%
}{\epsilon}e^{-\lambda \tau}d\tau-\langle
B^*\pi^*,v\rangle_{L^q_\lambda,L^p_\lambda}\right\vert= \\
&=\left\vert\int_t^{+\infty}\langle g_0^\prime(y(\tau;u^*+\epsilon_0 v))-
g_0^\prime(y^*(\tau)),\int_t^\tau e^{(\tau-\sigma)A}B
v(\sigma)d\sigma\rangle e^{-\lambda\tau}d\tau\right\vert \\
&\le [g_0^\prime]\epsilon_0\int_t^{+\infty}\left\vert \int_t^\tau
e^{(\tau-\sigma)A}B v(\sigma)d\sigma\right \vertv^2e^{-\lambda\tau}d\tau=:I
\end{split}%
\end{equation*}
We estimate now the right hand side in the case {when $p>2$ and $%
\lambda\not=\omega p$}. By H\"older inequality one has
\begin{equation*}
\begin{split}
\left\vert \int_t^\tau e^{(\tau-\sigma)A}B v(\sigma)d\sigma\right \vertv&\le
e^{\omega\tau}\left[\int_t^\tau e^{\left(\frac{\lambda}{p}%
-\omega\right)q\sigma} d\sigma\right]^{\frac{1}{q}}\Vert
B\Vert_{\mathcal L(U,V^\prime)}\vert v\vert_{L^p_\lambda(t,+\infty;U)} \\
&= e^{\omega\tau}\left\vert\frac{ e^{\left(\frac{\lambda}{p}%
-\omega\right)q\tau} -e^{\left(\frac{\lambda}{p}-\omega\right)qt}} {q\left(%
\frac{\lambda}{p}-\omega\right)}\right\vert^{\frac{1}{q}}\Vert
B\Vert_{\mathcal L(U,V^\prime)}\vert v\vert_{L^p_\lambda(t,+\infty;U)}
\end{split}%
\end{equation*}
Then, for a suitable constant $C_5$, depending on $g$, $B$ and $v$, we have
\begin{equation}  \label{I}
\begin{split}
I&\le [g_0^\prime]\epsilon_0\int_t^{+\infty} \left\vert
e^{\omega\tau}\left\vert\frac{ e^{\left(\frac{\lambda}{p} -\omega\right)q%
\tau} -e^{\left(\frac{\lambda}{p}-\omega\right)qt}} {q\left(\frac{\lambda}{p}%
-\omega\right)}\right\vert^{\frac{1}{q}}\Vert B\Vert_{\mathcal L(U,V^\prime)}\vert
v\vert_{L^p_\lambda(t,+\infty;U)}\right\vert^2 e^{-\lambda\tau}d\tau \\
&\le C_5 \epsilon_0\int_t^{+\infty} \left\vert e^{\left(\frac{\lambda}{p}%
-\omega\right)q\tau} -e^{\left(\frac{\lambda}{p}-\omega\right)qt}
\right\vert^{\frac{2}{q}} e^{-(\lambda-2 \omega) \tau}d\tau \\
&\le C_5\epsilon_0 \int_t^{+\infty} \left( e^{\left(\frac{\lambda}{p}
-\omega\right)q\tau}\vee 1\right)^\frac{2}{q}e^{-(\lambda-2\omega)\tau}d\tau
\\
&= C_5\epsilon_0 \int_t^{+\infty} e^{\lambda\left(\frac{2}{p}%
-1\right)\tau}\vee e^{-(\lambda-2\omega)\tau}d\tau. \\
\end{split}%
\end{equation}
Since $p>2\Rightarrow\lambda\left(\frac{2}{p}-1\right)<0,$ then one may let $%
\epsilon\to 0$ and obtains that the right hand side in (\ref{I}) goes to $0$.

{Let now $p>2$ and $\lambda=\omega p$. By H\"older inequality one has}
\begin{equation*}
\begin{split}
\left\vert \int_t^\tau e^{(\tau-\sigma)A}B v(\sigma)d\sigma\right \vertv
&\le \Vert B\Vert_{\mathcal L(U,V^\prime)} e^{\omega\tau}\int_t^\tau e^{-\frac{\lambda%
}{p}\sigma} \vert v(\sigma)\vert_U d\sigma \\
&\le \Vert B\Vert_{\mathcal L(U,V^\prime)}e^{\omega\tau}\vert \tau-t\vert^{\frac{1}{q}%
} \vert v\vert_{L^p_\lambda(t,+\infty;U)}
\end{split}%
\end{equation*}
Then, for suitable $C_6 >0$ we get
\begin{equation*}
\begin{split}
I&\le [g_0^\prime]\epsilon_0\int_t^{+\infty}\left\vert \Vert
B\Vert_{\mathcal L(U,V^\prime)}e^{\omega\tau}\vert \tau-t\vert^{\frac{1}{q}} \vert
v\vert_{L^p_\lambda(t,+\infty;U)} \right\vert v^2e^{-\lambda\tau}d\tau \\
&\le C_6\epsilon_0 \int_t^{+\infty} \vert \tau-t\vert^{\frac{2}{q}%
}e^{-(\lambda-2\omega)\tau}d\tau
\end{split}%
\end{equation*}
and by letting $\epsilon\to 0$ the right hand side goes to 0.

Let now $p=2$, $\lambda>0$ and $\omega <0$. Then we have
\begin{equation*}
\begin{split}
\left\vert \int_t^\tau e^{(\tau-\sigma)A}B v(\sigma)d\sigma\right \vertv^2
&\le \Vert B\Vert_{\mathcal L(U,V^\prime)}^2 \left[\int_t^\tau
e^{\omega(\tau-\sigma)}|v(\sigma)|d\sigma \right]^2 \\
&\le \Vert B\Vert_{\mathcal L(U,V^\prime)}^2 \frac{1}{\omega}\left[%
e^{\omega(\tau-t)}-1\right] \int_t^\tau e^{\omega(\tau-\sigma)}|v(\sigma)|^2
d\sigma
\end{split}%
\end{equation*}
where in the last step we used the Jensen's inequality. Hence, by using
Fubini-Tonelli Theorem, we get
\begin{equation*}
\begin{split}
&\int_t^{+\infty}\left\vert \int_t^\tau e^{(\tau-\sigma)A}B
v(\sigma)d\sigma\right \vertv^2 e^{-\lambda \tau}d\tau \le \\
&\le \Vert B\Vert_{\mathcal L(U,V^\prime)}^2 \int_t^{+\infty} \frac{1}{\omega}\left[%
e^{\omega(\tau-t)}-1\right] e^{-\lambda \tau} \int_t^\tau
e^{\omega(\tau-\sigma)}|v(\sigma)|^2 d\sigma d\tau = \\
&= \Vert B\Vert_{\mathcal L(U,V^\prime)}^2 e^{-\lambda t} \int_t^{+\infty}
e^{\omega(\sigma-t)}|v(\sigma)|^2 \int_\sigma^{+\infty} \frac{1}{\omega}%
\left[e^{\omega(\tau-t)}-1\right] e^{-(\lambda-\omega) (\tau-t)} d\tau
d\sigma = \\
&= \Vert B\Vert_{\mathcal L(U,V^\prime)}^2 e^{-\lambda t} \int_t^{+\infty}
|v(\sigma)|^2 \frac{1}{\omega}\left[ \frac{e^{-(\lambda-\omega)(\sigma-t)}}{%
\lambda-2\omega}- \frac{e^{-\lambda(\sigma-t)}}{\lambda-\omega} \right]
d\sigma \le \\
&\le \Vert B\Vert_{\mathcal L(U,V^\prime)}^2 \frac{1}{(-\omega)(\lambda-\omega)}
\int_t^{+\infty} |v(\sigma)|^2 e^{-\lambda\sigma} d\sigma
\end{split}%
\end{equation*}
where, in the last inequality, we used that $\omega<0$. This immediately
implies that
\begin{equation*}
I\le [g_0^\prime]\epsilon_0 \Vert B\Vert_{\mathcal L(U,V^\prime)}^2 \frac{1}{%
(-\omega)(\lambda-\omega)} \vert v\vert_{ L^2_\lambda(t,+\infty;U)}
\end{equation*}
which immediately gives the claim.

\smallskip

\emph{Claim 4. Assume that, either $p>2$ and $\lambda>0$, or $p=2$, $%
\lambda>0$ and $\omega <0$. Assume that $(u^*,y^*)$ is optimal at $(t,x)$,
and let $\pi^*$ be the associated solution of \eqref{pmild}. Then $%
(\pi^*,u^*,y^*)$ is a mild solution of \eqref{CO}.}


\smallskip

{We only need to prove that the last line of \eqref{CO}. From optimality of $%
u^*$ we have $\partial J(u^*)\ni 0$.} Then Claim 1 and Claim 3 imply $(\ref%
{mp2})$ and Claim 4 follows.

\hfill$\Box$

\medskip

\noindent\textbf{Proof of Theorem \ref{pi^*} }  Firstly we prove that $%
\pi^*(t;t,x)\in \partial {Z_0}(x)$. We recall that in Theorem \ref{Pmp} we
showed that $-B^*\pi^*(\tau;t,x)\in\partial h_0(u^*(\tau))$ almost
everywhere in $[0,+\infty)$. Then, for all $\bar x\in V^\prime$, and an
associated control $\bar u$, optimal at $(t,\bar x)$, we have
\begin{equation*}
\begin{split}
{Z_0}(\bar x)&- {Z_0}(x)= e^{\lambda t}\left[J (t,\bar x, \bar u)-J (t, x,
u^*)\right] \\
&\ge \int_t^{+\infty}\bigg[\langle g^\prime_0(y^*(\tau)), \bar
y(\tau)-y^*(\tau)\rangle -( B^*\pi^*(\tau;t,x))\vert \bar
u(\tau)-u^*(\tau))_U\bigg]e^{-\lambda(\tau-t)}d\tau. \\
\end{split}%
\end{equation*}
Note that
\begin{equation*}
\begin{split}
&\int_t^{+\infty}\left\langle g^\prime_0(y^*(\tau)), \bar
y(\tau)-y^*(\tau)\right\rangle_{V,V^\prime}e^{-\lambda(\tau-t)}d\tau= \\
&=\int_t^{+\infty}\left\langle g^\prime_0(y^*(\tau)), e^{A(\tau-t) }(\bar
x-x) +\int_t^\tau e^{A(\tau-\sigma)} B(\bar u(\sigma)-u^*(\sigma))d\sigma
\right\rangle e^{-\lambda(\tau-t)}d\tau= \\
&=\int_t^{+\infty}\left\langle
e^{(A_1^*-\lambda)(\tau-t)}g^\prime_0(y^*(\tau)), \bar x-x\right\rangle d
\tau+ \\
&\qquad + \int_t^{+\infty}\!\!\!\!\int_t^\tau \left\langle
g^\prime_0(y^*(\tau)), e^{A(\tau-\sigma)}\left(B( \bar
u(\sigma)-u^*(\sigma)\right)\right \rangle e^{-\lambda(\tau-t)}d\sigma d\tau.
\end{split}%
\end{equation*}
The last term can be rewritten exchanging the integrals as
\begin{equation*}
\begin{split}
&\int_t^{+\infty}\!\!\!\!\int_\sigma^{+\infty}\!\!\!\! \left(
B^*e^{(A_1^*-\lambda)(\tau-\sigma)} g^\prime_0(y^*(\tau))\vert \bar
u(\sigma)-u^*(\sigma)\right)_{U} e^{-\lambda(\sigma-t)}d\tau d\sigma = \\
=& \int_t^{+\infty}\!\!\!\! \left( B^*\pi^*(\sigma;t,x)\vert \bar
u(\sigma)-u^*(\sigma)\right)_{U} e^{-\lambda(\sigma-t)}d\sigma.
\end{split}%
\end{equation*}
Hence we get
\begin{equation*}
\begin{split}
{Z_0}(\bar x)&- {Z_0}( x)\ge\left \langle\int_t^{+\infty}
e^{(A_1^*-\lambda)(\tau-t)}g^\prime_0(y^*(\tau))d\tau, \bar x-x\right\rangle
=\langle \pi^*(t;t,x),\bar x-x\rangle,
\end{split}%
\end{equation*}
and the assertion is proven. The proof  {that $\pi^*(\tau;\tau,y^*(\tau))\in
\partial {Z_0} (y^*(\tau))$ for every $\tau\ge t$} 
is standard but we write it here for the sake of completeness. Let $\tau >t$
and observe that, by the dynamic programming principle, the control defined
by
\begin{equation*}
u_{0,y^*(\tau)}(\sigma)\equiv u^*(\sigma+\tau)
\end{equation*}
is optimal at $(0,y^*(\tau))$. Consequently the associated trajectory
satisfies
\begin{equation*}
y(\sigma;0,y^*(\tau), u_{0,y^*(\tau)})=y(\sigma+\tau; \tau,y^*(\tau),
u^*)=y^*(\sigma+\tau).
\end{equation*}
Then by the first part of the proof we have
\begin{equation*}
\begin{split}
\partial {Z_0}(y^*(\tau))&\ni\pi^*(0;0, y^*(\tau))= \\
&=\int_0^{+\infty}e^{(A_1^*-\lambda)\sigma}g_0^\prime(y(r;0,y^*(%
\tau),u_{0,y^*(\tau)}))d\sigma \\
&=\int_0^{+\infty}e^{(A_1^*-\lambda)\sigma}g_0^\prime(y^*(\sigma+\tau))d%
\sigma \\
&=\int_\tau^{+\infty}e^{(A_1^*-\lambda)(r-\tau)}g_0^\prime(y^*(r))dr \\
&=\pi^*(\tau;\tau,y^*(\tau))
\end{split}%
\end{equation*}
which gives the claim. It finally suffices to note that, when $p>2$, one has
by Theorem \ref{th:mainnew}, that $Z_0=\Psi$ and $\partial\Psi(x)=\{\Psi^%
\prime(x)\}$ to complete the proof. \hfill$\Box$

\medskip

\section{Proofs of Subsection \protect\ref{SS:EQPT}}

We here work out the proofs of the results stated in Subsection \ref{SS:EQPT}.

\smallskip

\noindent\textbf{Proof of Theorem \ref{th:equiv}.} We first prove $(i)$.
Assume $(\bar x,\bar \pi, \bar u)$ is a MP-equilibrium point for problem
(P).
Since $\lambda > \omega$, then $\lambda \in \rho(A_1^*)$, and the
second of \eqref{eq:PMP-1} applies, implying
\begin{equation}  \label{eq:barpig0}
\bar \pi=(\lambda-A_1^*)^{-1}g_0^\prime(\bar x).
\end{equation}
Then, plugging  \eqref{eq:barpig0} into the third equation of
\eqref{def:MPEP} and then $\bar u$ so obtained into the first, we derive \eqref{eqlts}.
Moreover, using \eqref{eq:barpig0} in  Theorem \ref{pi^*}, we
get \eqref{eq:MPisCLEweak}.

\medskip
We now prove $(ii)$. We consider a CLE-equilibrium point $\hat x$ and set $%
\hat u:=(h_0^*)^\prime(-B^*\Psi_0^{\prime }(\hat x))$. By Theorem \ref%
{th:uniquefeedback} and Remark \ref{rm:uniquefeedback} we know that the
couple $(\hat x, \hat u)$ is optimal.

By Theorem \ref{Pmp}-$(ii)$ we can associate, to such optimal couple, a
costate $\hat \pi$ which is the (mild) solution of the costate equation
in \eqref{CO}. Such mild solution is then necessarily stationary (see Definition \ref%
{df:mild}) and, since $\lambda>\omega$, given by
\begin{equation*}
\hat \pi \equiv (\lambda-A_0^*)^{-1}g_0^\prime(\hat x).
\end{equation*}
Moreover, by Theorem \ref{pi^*} it is also true that
$\hat \pi=\Psi^{\prime }(\hat x)$. As a consequence, $(\hat x, \hat \pi, \hat u)$ solves \eqref{def:MPEP}
and is then a MP-equilibrium point. \hfill$\Box$

Before demonstrating Lemma \ref{fixedpoint} we need to state and prove the
following result.

\begin{prop}
Assume that $0\in \rho(A_1^*)$ (that is, $(A_1^*)^{-1}$ is well defined and
bounded in $H$). Then $A^{-1}$ has bounded inverse on $V^\prime$, defined by
the position
\begin{equation*}
\langle A^{-1}f,\varphi\rangle = \langle
f,(A_1^*)^{-1}\varphi\rangle, \ for\ all\ f\in V^\prime\ and\
\varphi\in V.
\end{equation*}
Moreover
\begin{equation*}
\Vert A^{-1}\Vert_{\mathcal L(V^\prime)}\le \Vert (A_1^*)^{-1}\Vert_{\mathcal L(H)}.
\end{equation*}
\end{prop}

\begin{proof}
For all $f\in V^\prime$ and $\varphi\in V$ we have
\begin{equation}
\begin{split}
\vert \langle A^{-1}f,\varphi\rangle \vert&=\vert \langle
f,(A_1^*)^{-1}\varphi\rangle \vert \\
&\le\vert f\vert_{_{V^\prime}}\vert(A_1^*)^{-1}\varphi\vert_V \\
&=\vert f\vert_{_{V^\prime}}(\vert(A_1^*)^{-1}\varphi\vert_H+\vert
A_1^*(A_1^*)^{-1}\varphi\vert_H) \\
&=\vert f\vert_{_{V^\prime}}(\vert(A_1^*)^{-1}\varphi\vert_H+
\vert(A_1^*)^{-1}A_1^*\varphi\vert_H) \\
&\le\vert
f\vert_{_{V^\prime}}\Vert(A_1^*)^{-1}\Vert_{\mathcal L(H)}\vert\varphi\vert_V
\end{split}%
\end{equation}
\end{proof}

\noindent\textbf{Proof of Lemma \ref{fixedpoint}.}

Define $T: V^\prime\to V^\prime$ as as in (\ref{Tlambda}). By its
definition, $T$ satisfies
\begin{equation}
\begin{split}
\vert Tx-Ty\vert_{_{V^\prime}}&\le \Vert (A)^{-1}\Vert_{\mathcal L(V^{\prime })}\Vert
B\Vert^2_{\mathcal L(U,V^\prime)} [(h_0^*)^\prime][g_0^\prime]\frac{1}{\lambda-\omega}
\vert x-y\vert_{_{V^\prime}}.
\end{split}%
\end{equation}
implying the claim. Since $T(V^{\prime })\subseteq D(A)$, then fixed points
of $T$ lie in $D(A)$. \hfill$\Box$

\medskip

\noindent \textbf{Proof of Proposition \ref{pr:stabilitybis}} Let $x\in I(%
\bar{x})$ and let $x^{\ast }(t)$ be the associated optimal trajectory, i.e.
the solution of the closed loop equation
\begin{equation}
\begin{cases}
y^{\prime }(t)=Ay(t)+f(y(t)), & t>0 \\
y(0)=x, &
\end{cases}
\label{CLEfbis}
\end{equation}%
Let $x_{n}\in I(\bar{x})\cap D(A)$ be such that $x_{n}\rightarrow x$ in $%
V^{\prime }$ as $n\rightarrow +\infty $. Let $y_{n}^{\ast }(t)$ be the
associated optimal trajectory. Since $y_{n}^{\ast }$ is continuous, then it
must remain in $I(\bar{x})$ at least for a sufficiently small $t$. For such $%
t$ we must have (recall that $y_{n}^{\ast }(t)\in D(A)$ by assumption)
\begin{equation*}
\frac{1}{2}\frac{d}{dt}|y_{n}^{\ast }(t)-\bar{x}|_{V^{\prime }}^{2}=\big(
(y_{n}^{\ast }(t))^{\prime }\vert y_{n}^{\ast }(t)-\bar{x}\big) _{V^{\prime }}
\end{equation*}%
\begin{equation*}
=\big( A(y_{n}^{\ast }(t)-\bar{x})+f(y_{n}^{\ast }(t))-f(\bar{x}%
)\vert y_{n}^{\ast }(t)-\bar{x}\big) _{V^{\prime }}\leq \xi |y_{n}^{\ast }(t)-%
\bar{x}|_{V^{\prime }}^{2}.
\end{equation*}%
This implies that
\begin{equation}
|y_{n}^{\ast }(t)-\bar{x}|_{V^{\prime }}^{2}\leq e^{2\xi t}|x_{n}-\bar{x}%
|_{V^{\prime }}^{2}  \label{eq:xndiss}
\end{equation}

Next we take the limits as $n\rightarrow +\infty $. Note that $%
z_{n}(t):=y_{n}^{\ast }(t)-x^{\ast }(t)$ solves
\begin{equation}
\begin{cases}
z_{n}^{\prime }(t)=Az_{n}(t)+f(y_{n}^{\ast }(t))-f(x^{\ast }(t)), & t>0 \\
z(0)=x_{n}-x. &
\end{cases}
\label{CLEappr}
\end{equation}%
By Theorem \ref{th:uniquefeedback} $f$ is Lipschitz continuous so with
Lipschitz constant $[f]_{0,1}$, a standard Gronwall inequality implies
\begin{equation*}
|y_{n}^{\ast }(t)-x^{\ast }(t)|_{V^{\prime }}\leq |x_{n}-x|_{V^{\prime
}}e^{(\omega +[f]_{0,1})t},
\end{equation*}

hence $|y_{n}^{\ast }(t)-x^{\ast }(t)|_{V^{\prime }}$ converges to $0$ for
every $t\geq 0$. Consequently, from \eqref{eq:xndiss} one has
\begin{equation*}
|x^{\ast }(t)-\bar{x}|_{V^{\prime }}^{2}\leq e^{2\xi t}|x-\bar{x}%
|_{V^{\prime }}^{2},\ \ \text{ for\ every\ }t\geq 0,
\end{equation*}%
implying the claims. \hfill $\Box $

\medskip

\noindent \textbf{Proof of Corollary \ref{cor:Adiss}.} It is enough to check
that \eqref{eq:AdissV'} is verified, either in $I$, or in $V^{\prime }$. If $%
f$ is Lipschitz continuous on $I$ (in the topology of $V^{\prime }$), with
Lipschitz constant $\bar{\theta}<\theta $ then we have
\begin{equation*}
(f(x)-f(\bar{x})|x-\bar{x})_{V^{\prime }}\leq \bar{\theta}|x-\bar{x}%
|_{V^{\prime }}^{2},\ \ \forall x\in I.
\end{equation*}%
This immediately implies that \eqref{eq:AdissV'} holds in $I$ with  $\xi =%
\bar{\theta}-\theta $.

Similarly, if $f$ is globally Lipschitz continuous in $V^{\prime }$, we get
that \eqref{eq:AdissV'} is verified in the whole $V^{\prime }$. \color{black}

\hfill$\Box$

\section{Proofs of Section \protect\ref{sec:avc}}\label{AC}
\medskip

\noindent \textbf{Proof of Lemma \ref{lm:fxptvintage}.}
Note that if $A_0^*$ and  $A$ are  the operators
described in Section \ref{INTROSUBMAT}, then the following two facts hold true. First,  by definition of $A_0^*$, we have 
$\bar{\alpha}=(\lambda -A_{0}^{\ast })^{-1}\alpha$. Second, $A$ is invertible, so that    equation 
\eqref{eqlts} may be rewritten as $Tx=x$ by means of the operator $T$ defined in (\ref{Tlambda}).
Now (\ref{dati}) holds, so that one has
\begin{equation*}
g_{0}^{\prime }(x)=-R^{\prime }(\left\langle \alpha ,x\right\rangle )\alpha
,\ \text{\ and\ \ } (\lambda -A_{0}^{\ast })^{-1}g_{0}^{\prime
}(x)=-R^{\prime }(\left\langle \alpha ,x\right\rangle )\overline{\alpha }
\end{equation*}%
Moreover, since $h_{0}(u)=C(u_{0},u_{1})=C_{0}(u_{0})+C_{1}(u_{1}),$ then
the convex conjugate $C^{\ast }$ of $C$ is also of type $C^{\ast
}(u_{0},u_{1})=C_{0}^{\ast }(u_{0})+C_{1}^{\ast }(u_{1}).$ Then, recalling
the definition of $B,B^{\ast }$ in Section \ref{MODEL}, and by means of (\ref%
{alphabar}), the operator $T$ defined in (\ref{Tlambda}) can be rewritten as
follows:
\begin{eqnarray*}
Tx &=&-A^{-1}B(C^{\ast })^{\prime }\left( B^{\ast }R^{\prime }(\left\langle
\alpha ,x\right\rangle )\overline{\alpha }\right) \\
&=&-A^{-1}B(C^{\ast })^{\prime }\left( R^{\prime }(\left\langle \alpha
,x\right\rangle )\overline{\alpha }(0),R^{\prime }(\left\langle \alpha
,x\right\rangle )\overline{\alpha }\right) \\
&=&-A^{-1}B\left( \left( C_{0}^{\ast }\right) ^{\prime }\left( R^{\prime
}(\left\langle \alpha ,x\right\rangle )\overline{\alpha }(0)\right) ,\left(
C_{1}^{\ast }\right) ^{\prime }\left( R^{\prime }(\left\langle \alpha
,x\right\rangle )\overline{\alpha }\right) \right) \\
&=&-A^{-1}\left[ \left( C_{0}^{\ast }\right) ^{\prime }\left( R^{\prime
}(\left\langle \alpha ,x\right\rangle )\overline{\alpha }(0)\right) \delta
_{0}+\left( C_{1}^{\ast }\right) ^{\prime }\left( R^{\prime }(\left\langle
\alpha ,x\right\rangle )\overline{\alpha }\right) \right] \\
&=&\left( C_{0}^{\ast }\right) ^{\prime }\left( R^{\prime }(\left\langle
\alpha ,x\right\rangle )\overline{\alpha }(0)\right) [-A^{-1}\delta
_{0}]+[-A^{-1}]\left( C_{1}^{\ast }\right) ^{\prime }\left( R^{\prime
}(\left\langle \alpha ,x\right\rangle )\overline{\alpha }\right)
\end{eqnarray*}%
which by (\ref{A-1}) implies (\ref{Texpl}). By Theorem \ref{th:equiv}
we derive the remaining statements. \hfill $\Box $

\medskip

\noindent \textbf{Proof of Theorem \ref{equgen}.} Set $\eta:V^{\prime }\to \mathbb{R}$, $\eta(x) =R^{\prime
}(\left\langle \alpha ,x\right\rangle )$ and note that, by \eqref{Texpl}, we
get $Tx(s)=F(\eta (x))(s)$. Then the equation $Tx=x$ is rewritten as
\begin{equation*}
F(\eta (x))[s]=x(s), \qquad \forall s \in[0,\bar s].
\end{equation*}
Applying $\eta$ on both sides of such equation we get
\begin{equation*}
R^{\prime }(\left\langle \alpha ,F(\eta (x))\right\rangle ) = \eta (x).
\end{equation*}
Hence, if $\bar x$ is a solution of $Tx=x$, then $\eta(\bar x)$ is a
solution of \eqref{eta=R}. Viceversa, let $\bar{\eta }$ be a solution to %
\eqref{eta=R}. Then, substituting into \eqref{eq:defFeta} and \eqref{Texpl},
we get that $F(\bar{\eta })$ solves $F(\bar{\eta })=T F(\bar{%
\eta })$, so $F(\bar{\eta })$ solves $Tx=x$.  \hfill $\Box $

\medskip
\noindent \textbf{Proof of Theorem \ref{stablin}.}
Note that when  the value function $V_0$ is differentiable, then
Proposition \ref{pr:stability}  applies with $\Psi=-V_0$.
Indeed by Lemma \ref{lm:verassvintage} and   the Lumer-Philips
Theorem (see e.g. \cite[Theorem 1.4.3]{Pazy}),   the operator $A$  is
dissipative in $V^{\prime }$ with $( Ax\vert x)_{V^{\prime }}\le -\mu
|x|_{V'}^2$, and the function $f$ defined  in Proposition \ref{pr:stability}  can be rewritten by means of  \eqref{B^*} \eqref{ipo3}  as
\begin{eqnarray*}
f(x)
&=&-B\left( \frac{1}{4\beta _{0}}\left(  \Psi ^{\prime
}(x)[0]  -q_{0}\right) ,\frac{1}{4\beta _{1}\left( \cdot \right) }\left(
\Psi ^{\prime }(x)-q_{1}\right) \right) \\
&=&-\frac{\langle \delta _{0},\Psi ^{\prime }(x)\rangle-q_{0}}{4\beta _{0}}\;\delta _{0}-%
\frac{1}{4\beta _{1}\left( \cdot \right) }\left( \Psi ^{\prime
}(x)-q_{1}\right).
\end{eqnarray*}
Hence, if $\bar x$ is a CLE-equilibrium point, for all $x\in V^{\prime },$ we have
\begin{equation*}
\left( f(x)-f(\bar x)\vert x-\bar x\right)_{V^{\prime }} = - \frac{\langle\delta _{0},\Psi
^{\prime }(x)-\Psi ^{\prime }(\bar x)\rangle}{4\beta _{0}}\;(\delta
_{0}\vert x-\bar x)_{V^{\prime }}- \left( \frac{%
1}{4\beta _{1}\left( \cdot \right) }\left( \Psi ^{\prime }(x)-\Psi ^{\prime
}(\bar x)\right)\bigg\vert x-\bar x\right)_{V^{\prime }}.
\end{equation*}
Note that the second term in the above inequality satisfies
$$\left( \frac{%
1}{4\beta _{1}\left( \cdot \right) }\left( \Psi ^{\prime }(x)-\Psi ^{\prime
}(\bar x)\right)\bigg\vert x-\bar x\right)_{V^{\prime }}\ge \frac 1
{4\vert \beta_1\vert_{L^\infty(0,\bar s)}}\left(\Psi ^{\prime }(x)-\Psi ^{\prime
}(\bar x)\vert x-\bar x\right)_{V^{\prime }}\ge 0$$
since $\beta _{1}\in L^\infty(0,\bar s)$, and $\Psi^\prime $ is a monotone
operator (as a consequence of the convexity of $\Psi$). That in particular
implies
\begin{equation}  \label{dissstima}
\left( f(x)-f(\bar x)\vert x-\bar x\right)_{V^{\prime }} \le -\frac{\langle\delta
_{0},\Psi ^{\prime }(x)-\Psi ^{\prime }(\bar x)\rangle}{4\beta _{0}}\;( \delta
_{0}\vert x-\bar x)_{V^{\prime }}\le \frac {[\Psi]_L \vert \delta
_{0}\vert_{V^{\prime }}^2} {4\beta_0}\vert x-\bar x\vert^2_{V^{\prime }},
\end{equation}
so that \eqref{eq:AdissV'} is satisfied with $\xi=-\mu+  {[\Psi]_L \vert \delta
_{0}\vert_{V^{\prime }}^2} /({4\beta_0})$.
 \hfill $\Box $

\end{document}